\documentclass[a4paper, 11pt]{article}
\usepackage{mathrsfs}
\usepackage{bbm}
\usepackage{}
\usepackage{amsfonts}
\usepackage{mathrsfs,color}
\usepackage{graphicx,amsmath,amsthm,amssymb,amsfonts}
\usepackage{subcaption}
\usepackage{booktabs}
\usepackage[left=1.5cm,right=1.5cm,top=2cm,bottom=2cm]{geometry}
\usepackage{float}

\DeclareFontFamily{OT1}{pzc}{}
\DeclareFontShape{OT1}{pzc}{m}{it}{<-> s * [1.200] pzcmi7t}{}
\DeclareMathAlphabet{\mathpzc}{OT1}{pzc}{m}{it}

\DeclareMathOperator{\SO}{SO}

\DeclareMathOperator{\so}{\mathfrak{so}}
\DeclareMathOperator{\su}{\mathfrak{su}}

\DeclareMathOperator{\gl}{\mathfrak{gl}}
\DeclareMathOperator{\sym}{\mathfrak{sym}}

\newcommand{\R}{\mathbb{R}}

\newtheorem{theorem}{Theorem}
\newtheorem{definition}{Definition}

\newtheorem{lemma}{Lemma}

\newtheorem{proposition}{Proposition}

\DeclareFontFamily{OT1}{pzc}{}
\DeclareFontShape{OT1}{pzc}{m}{it}{<-> s * [1.200] pzcmi7t}{}
\DeclareMathAlphabet{\mathpzc}{OT1}{pzc}{m}{it}

\newcommand*\mcap{\mathbin{\mathpalette\mcapinn\relax}}
\newcommand*\mcapinn[2]{\vcenter{\hbox{$\mathsurround=0pt
  \ifx\displaystyle#1\textstyle\else#1\fi\bigcap$}}}

\newcommand*\mcup{\mathbin{\mathpalette\mcupinn\relax}}
\newcommand*\mcupinn[2]{\vcenter{\hbox{$\mathsurround=0pt
  \ifx\displaystyle#1\textstyle\else#1\fi\bigcup$}}}

\begin{document}

\title{\bf Structural Controllability  on Graphs for\\  Drifted Bilinear Systems over Lie Groups}

\date{}

\author{Xing Wang\thanks{Key Laboratory of Mathematics Mechanization, Institute of Systems Science, Academy of Mathematics and Systems Science, Chinese Academy of Sciences, Beijing 100190, China; School of Mathematical Sciences, University of Chinese Academy of Sciences, Beijing 100049, China. (wangxing17@amss.ac.cn)}, Bo Li\thanks{Key Laboratory of Mathematics Mechanization, Academy of Mathematics and Systems Science, Chinese Academy of Sciences, Beijing 100190, China. (libo@amss.ac.cn)}, Jr-Shin Li\thanks{Department of Electrical and Systems Engineering, Washington University, St. Louis, MO 63130, USA. (jsli@wustl.edu)},
 	Ian R. Petersen\thanks{Research School of Engineering, Australian National University, Canberra, ACT 0200, Australia. (ian.petersen@anu.edu.au)}, Guodong Shi\thanks{Australian Center for Field Robotics, School of Aerospace, Mechanical and Mechatronic Engineering, The University of Sydney, NSW 2006, Australia. (guodong.shi@sydney.edu.au)}}
	
\maketitle

\begin{abstract}
In this paper, we study graphical conditions for structural controllability and accessibility  of drifted bilinear systems over Lie groups. We consider a bilinear control system with   drift  and  controlled terms that evolves over the special orthogonal group, the general linear group, and the special unitary group. Zero patterns are prescribed for the drift and controlled dynamics with respect to a set of base elements in the corresponding Lie algebra. The drift dynamics must respect a rigid zero-pattern in the sense that the drift takes values as a linear combination of base elements with strictly non-zero coefficients; the controlled dynamics are allowed to follow a free zero pattern with potentially zero coefficients in the configuration of the controlled term by linear combination of the controlled base elements. First of all, for such bilinear systems over  the special orthogonal group or the special unitary group, the zero patterns are shown to be associated with two undirected or directed graphs whose connectivity and connected components  ensure structural controllability/accessibility. Next, for bilinear systems over the special unitary group, we introduce two edge-colored graphs associated with the drift and controlled zero patterns, and prove structural controllability conditions related to connectivity and the number of edges of a particular color. 
\end{abstract}
 
\section{Introduction}
The past decade has witnessed an important line of   research for multi-agent systems where {\em agents}, representing subsystems with distributed sensing and control units, are dynamically  interconnected over an underlying  network in order to achieve collective goals	such as consensus, formation, coverage, and controllability \cite{jad03,mesbahi}. The links of such a network might indicate physical interactions, or non-physical information exchanges, which define a {\em topology} of the overall multi-agent system.  The interface of classical control theory and graph theoretic methods has led to 
fundamental  insights regarding how the network topology enables convergence of distributed algorithms, stabilizability of distributed controllers, or controllability \cite{Murray04,Martinez07,Nedic10,C2015,B2007,Rahmani09} of the network dynamical states, for both linear and nonlinear network dynamics. 

Perhaps one  central problem in networked systems is the ability to control the system dynamics with control inputs scattered amongst a subset of agents (nodes), leading to a natural network controllability problem. Indeed classical control theories have established celebrated results on conditions for  controllability of general dynamical systems. In the networked system era, graphical conditions that link notions of {\em connectivity} for the network topology to the controllability of the system are desired.  Sparked by such a vision, a series of important understandings on graph-theoretic controllability  were established for networked systems with linear dynamics \cite{Rahmani09,ji2009,Parlangeli12,Gharesifard17,Chen15}. Recently such developments have been pushed further to  structural controllability \cite{Morse2019,Trentelman2020,Trentelman2021}, where instead of focusing on a specific system setup, controllability is defined as a structural property for generic configurations of the system parameters. The seminal work of Lin established that structural controllability is fully determined by the system {\em zero-pattern}, which specifies the  locations of the potentially non-zero entries in the parameter space. It turns out that the structural controllability of linear networked systems can also be established from a graph-theoretic point of view, marking  significant progress in the study of network controllability \cite{Morse2019,Trentelman2020,Trentelman2021}.

Controllability analysis  from graph theory is a challenging question for  multi-agent systems with general nonlinear dynamics. For systems with bilinear dynamics, however,  graph-theoretic insights have been shown to be promising towards understanding  controllability conditions \cite{structural,li2017}. In bilinear systems, the dynamical evolution of the system state is governed by the product of the system state and the control inputs. Therefore, bilinear systems are nonlinear systems with special geometric constraints and the bilinear property: the system state evolves in a Lie group, and  the vector field is a bilinear function taking values in the corresponding Lie algebra with respect to the state and the input  \cite{Elliott2009}. Despite its simplicity,   bilinear systems have a wide range of applications  in the fields of engineering, economic, and even quantum systems \cite{Khapalov1996,altafini2002,albertini2003,dir08,alb02,qi2020}; the study of controllability for bilinear system over Lie groups led to the discovery of the  Lie algebra rank condition \cite{Elliott1971,Brockett1973,Aoki1975,H1972,J1975,W1975} as an important chapter in nonlinear system theory.

The Lie algebra rank condition has been shown to be related to the graphical perspective as well. In \cite{structural}, a framework for the structural controllability of bilinear control systems was proposed, where it was shown that the connectivity of the underlying interaction graph determines the structural controllability of several classes of bilinear control systems over matrix groups. In \cite{li2017}, a graphical notion of permutation cycles was introduced to bilinear systems for the characterization of controllability over the special orthogonal group. Although the result in \cite{structural} considered a special class of drifted systems, the results derived in \cite{structural,li2017} mostly focused on driftless systems. Later in \cite{wang2020,arxiv}, controllability and accessibility conditions for bilinear systems were extended to drifted dynamics, where the drift dynamics and controlled dynamics define two interaction graphs, respectively, and their joint connectivity was shown to be critical for ensuring controllability or accessibility of systems over different matrix Lie groups. It is worth emphasizing  that for linear networked systems, the system graphs are   motivated from agent interconnections often with a physical implication  \cite{Rahmani09,ji2009,Parlangeli12,Gharesifard17,Chen15}, while for bilinear systems the system graphs are artificially introduced in order to facilitate a graphical analysis.

In this paper, we study graph theoretic conditions  for the structural controllability and accessibility of drifted  bilinear systems over Lie groups.  For a drifted bilinear system over the special orthogonal group, the general linear group, or  the special unitary group, zero patterns are introduced for the drift dynamics and the controlled dynamics, respectively. The zero pattern associated with the drift is assumed to be rigid, as a linear combination with non-zero coefficients of several base elements in a finite  subset of the corresponding Lie algebra.  As a result, the parameters of the drift dynamics are possibly  dependent, similar to the setup in \cite{Morse2019} for linear networked systems. The zero pattern associated with the controlled terms is assumed to be free, where coefficients might take zero values for the linear combination of the base elements. First of all, for systems over the special orthogonal group and the general Lie group, two (undirected or directed) graphs are introduced, respectively, corresponding to the drift zero pattern and the controlled zero pattern;  structural  controllability/accessibility conditions are derived based on connectivity and sizes of connected components. Next, for systems over the special unitary group, we introduced two edge-colored graphs corresponding to the zero patterns of the drift and controlled dynamics; and then conditions based on connectivity and number of edges in a particular color are established for the structural controllability of the system.

Along with the theoretical results, various examples are presented to illustrate the intuitions behind. The technical proofs for all statements are put in the Appendix.  The remainder of the paper is organized as follows. In Section \ref{sec:pre}, we present some preliminary material  in graph theory. In Section \ref{sec:problem}, we define the problem of interest. In Section \ref{sec:results}, the main results of the work are presented, followed by some concluding remarks in  Section \ref{sec:conc}.

\section{Graph Theory Preliminaries}\label{sec:pre}
\subsection{Undirected Graph}
An undirected graph  $\mathrm
{G}=(\mathrm {V}, \mathrm {E})$ consists of a finite set
$\mathrm{V}$ of nodes and an edge set
$\mathrm {E}$, where $\mathrm {E}$  is a set of unordered pairs of elements in ${\rm V}$. If there is an edge $\{v_i,v_j\}\in\mathrm {E}$,
we say that $v_i$ and $v_j$ are {\it adjacent} or {\it neighbours}. The number of neighbours of node $v$ is called the {\it degree} of $v$, denoted by ${\rm deg}(v)$. A graph $\mathrm{G}$ is a {\it bi-graph} if there is a partition of the node set into $\mathrm{V}=\mathrm{V}_{1}\bigcup \mathrm{V}_{2}$ with $\mathrm{V}_{1}$ and $\mathrm{V}_{2}$ being nonempty and mutually disjoint, where all edges are between $\mathrm{V}_{1}$ and $\mathrm{V}_{2}$.

A {\it path} between two nodes $v_1$ and $v_k$ in $\mathrm{G}$ is a sequence of distinct nodes
$v_1v_2\dots v_{k}$ such that for any $m=1,\dots,k-1$, there is an edge between $v_m$ and $v_{m+1}$. A pair of distinct  nodes $v_i$ and $v_j$
is said to be {\it reachable} from each other if there is a path between them.  A node is always assumed to be reachable
from itself. We call a graph $\mathrm{G}$ {\it connected} if every pair of distinct nodes in $\mathrm{V}$ is reachable from each other. The subgraph of $\mathrm{G}$ derived from the node set $\mathrm{V}^\ast \subseteq \mathrm{V}$, denoted by $\mathrm{G}|_{\mathrm{V}^\ast}$,
 is the graph $(\mathrm{V}^\ast, \mathrm{E}^\ast)$, where $\{v_i,v_j\}\in \mathrm{E}^\ast$ if and only if $\{v_i,v_j\}\in \mathrm{E}$ for $v_i,v_j\in \mathrm{V}^\ast$. A {\it connected component} (or just component) of  $\mathrm
{G}$ is a connected subgraph induced by some $\mathrm{V}^\ast \subseteq \mathrm{V}$, which is reachable to no additional nodes in $\mathrm
{V}\setminus \mathrm{V}^\ast$. Let $\mathrm{G}_1=(\mathrm {V}_1, \mathrm {E}_1)$, $\mathrm{G}_2=(\mathrm {V}_2, \mathrm {E}_2)$. The union graph $\mathrm{G}_1\mcup\mathrm{G}_2$ is defined by $\mathrm{G}_1\mcup\mathrm{G}_2=(\mathrm{V}_1\mcup\mathrm{V}_2,\mathrm{E}_1\mcup\mathrm{E}_2)$.

\subsection{Directed Graph}
A {\it directed} graph (digraph) $\mathcal
{G}=(\mathrm {V}, \mathcal {E})$ consists of a finite set
$\mathrm {V}$ of nodes and an arc set
$\mathcal {E}\subseteq \mathrm{V}\times\mathrm{V}$, where  $(v_i,v_j)\in\mathcal {E}$ denotes an
{\it arc}  from node $v_i\in \mathrm{V}$  to node $v_j\in\mathrm{V}$.
For $(v_i,v_j)\in\mathcal {E}$, we say that $v_i$ is an {\it in-neighbor} of $v_j$ and $v_j$ is an {\it out-neighbor} of $v_i$.
The number of in-neighbors and out-neighbors of $v$ is called its {\it in-degree} and {\it out-degree}, denoted by ${\rm deg}^{+}(v)$ and ${\rm deg}^{-}(v)$, respectively.

A {\it self-loop} in a digraph is an arc starting from and pointing to the same node. A digraph $\mathcal{G}$ is {\it simple} if it has no self-loops. $\mathcal{G}$ is {\it simple complete} if $\mathcal {E}=\mathrm{V}\times\mathrm{V}\setminus \{(v_i,v_i): v_i\in \mathrm{V}\}$.
The digraph obtained by removing the self-loop of $\mathcal{G}$ is called the simple digraph corresponding to $\mathcal{G}$. A {\it directed path} from node $v_1\in \mathrm{V}$ to $v_k\in \mathrm{V}$ is a sequence of distinct nodes $v_1v_2\dots v_{k}$
such that for any $m=1,\dots,k-1$, $(v_m, v_{m+1})$ is a directed arc in $\mathcal{E}$.
 We say that node $v_j$ is {\it reachable} from node $v_i$ if there is a directed path from $v_i$ to $v_j$. A digraph $\mathcal{G}$ is {\it strongly connected} if every two  nodes are mutually reachable.  A {\it weakly connected component} of a digraph $\mathcal{G}$ is a component of $\mathcal{G}$ when the directions of links are ignored. Let $\mathcal{G}_1=(\mathrm {V}_1, \mathcal {E}_1)$, $\mathcal{G}_2=(\mathrm {V}_2, \mathcal {E}_2)$. The union graph $\mathcal{G}_1\mcup\mathcal{G}_2$ is defined by $\mathcal{G}_1\mcup\mathcal{G}_2=(\mathrm{V}_1\mcup\mathrm{V}_2,\mathcal{E}_1\mcup\mathcal{E}_2)$.
\subsection{Edge-colored Multigraph}
Let $\textbf{C}=\{1,2,\dots,k\}$ be a set with cardinality $k$, whose elements are
called colors. An {\it edge-colored multigraph} $\mathscr{G}=(\mathrm {V}, \mathscr {E})$ consists of a finite set $\mathrm{V}$ of nodes and an edge set
$\mathscr {E}$, where an element $\mathbbm{e}=\big\{{i,j};c\big\}\in\mathscr {E}$ denotes an edge with color $c\in \textbf{C}$ between nodes $i\in{\rm V}$ and $j\in{\rm V}$ \cite{Groos2006}. Each edge has a set of one or two nodes associated to it, which are called its endpoints. A {\it self-loop} $\{i,i;c\}\in\mathscr {E}$ is an edge that joins a single endpoint to itself. A {\it multi-edge} is a collection of two or more edges having identical endpoints. The edges in a multi-edge are distinguished by different colors. We call an edge-colored multigraph {\it simple} if it has neither self-loops nor multi-edges. A {\it complete} edge-colored multigraph is a simple edge-colored multigraph such that every pair of nodes is
joined by an edge. We say that an edge-colored multigraph $\mathscr{G}$ is {\it empty} if  $\mathscr{E}$ is empty.

In an edge-colored multigraph $\mathscr{G}$, a {\it walk} $W=\langle v_0,\mathbbm{e}_1,v_1,\mathbbm{e}_2,\dots,v_{k-1},\mathbbm{e}_k,v_k\rangle$ from node $v_0$ to node $v_k$ is an alternating sequence of nodes and edges, such that the endpoints of $ \mathbbm{e}_i$ is equal to $\{v_{i-1},v_i\}$, for $i = 1,\dots,k$.  A {\it closed walk} (or cycle) is a walk that begins and ends at the same node. Node $v_j$ is {\it reachable} from node $v_i$ if there is a walk from $v_i$ to $v_j$. An edge-colored multigraph is {\it connected} if for every pair of nodes $v_i $ and $v_j$, there is a walk
from $v_i$ to $v_j$.  A {\it subgraph} of $\mathscr{G}$ is an edge-colored multigraph $\mathscr{H}$ whose nodes and edges are all
in $\mathscr{G}$. A maximal connected subgraph of $\mathscr{G}$ is called a {\it connected component} or simply
a component of $\mathscr{G}$. Let $\mathscr{G}_1=(\mathrm {V}_1, \mathscr {E}_1)$, $\mathscr{G}_2=(\mathrm {V}_2, \mathscr {E}_2)$. The union graph $\mathscr{G}_1\mcup\mathscr{G}_2$ is defined by $\mathscr{G}_1\mcup\mathscr{G}_2=(\mathrm{V}_1\mcup\mathrm{V}_2,\mathscr{E}_1\mcup\mathscr{E}_2)$.

\section{Problem Definition}\label{sec:problem}
\subsection{Bilinear Control Systems over Lie Groups}
Let $\mathbf{G}$ be a connected Lie group and $\mathfrak{g}$ be its corresponding Lie algebra. We consider the following bilinear control system over $\mathbf{G}$:
\begin{align}\label{bilinear}
	\dot{X}(t)= \mathsf{B}_0 X(t)+\Big(\sum_{i=1}^m u_i(t) \mathsf{B}_i\Big) X(t), \ \ X(0)=X_{0},
\end{align}
where $X(t)\in \mathbf{G}, \mathsf{B}_i\in \mathfrak{g}$ for $i=0,\dots,m$, and $u_i(t) \in \mathbb{R}$ are piecewise continuous  control signals for $i=1,\dots,m$.
For $T\geq 0$, the set $\mathcal{R}_{T}(X_{0})$ consists of the points in $\mathbf{G}$ that are {\it attainable} from $X_{0}$ at time $T$; i.e., all terminal points $X(T)$ of solutions of system \eqref{bilinear} originating at $X(0)=X_{0}$ under all possible $u_i(t), i=1,\dots,m,t\in[0,T]$. The {\it attainable set} $\mathcal{R}(X_{0})$ then is defined as the union of such sets $\mathcal{R}_{T}(X_{0})$ for all $T\geq0$;  i.e.,
$\mathcal{R}(X_{0}):=\bigcup_{T\geq0}\mathcal{R}_{T}(X_{0})\subset\mathbf{G}.$
Let $I$ be the identity of $\mathbf{G}$.

\begin{definition}\label{def1}{\rm(\cite{dir08})}
	The system \eqref{bilinear} is called accessible  if   $\mathcal{R}(I)$ has an interior point in $\mathbf{G}$; and controllable, if $\mathcal{R}(I)=\mathbf{G}$.
\end{definition}

The system Lie algebra of $\eqref{bilinear}$ is given as $\{\mathsf{B}_0,\mathsf{B}_1,\dots,\mathsf{B}_m\}_{\rm LA}$,
where $\{\mathsf{B}_0,\mathsf{B}_1,\dots,\mathsf{B}_m\}_{\rm LA}$ is the generated Lie subalgebra of  $\mathsf{B}_0,\mathsf{B}_1,\dots,\mathsf{B}_m$.
The algebraic criteria developed in \cite{H1972, J1975, dir08} can be used to verify the accessibility  and controllability of the system $\eqref{bilinear}$ by exploiting the algebraic structure of the system Lie algebra. That is, the system $\eqref{bilinear}$ is accessible on the Lie group $\mathbf{G}$ if and only if the system Lie algebra satisfies $\{\mathsf{B}_0,\mathsf{B}_1,\dots,\mathsf{B}_m\}_{\rm LA}=\mathfrak{g}$; If $\mathsf{B}_0=0$ or the Lie group $\mathbf{G}$ is compact, then the system $\eqref{bilinear}$ is controllable on the Lie group $\mathbf{G}$ if and only if it is accessible on the Lie group $\mathbf{G}$.

\subsection{Zero Patterns}
Let us assume that $\mathfrak{g}$ is a finite dimensional vector space over $\mathbb{R}$. Let  $\mathfrak{g}_{a}:=\{\textrm{g}_{a_1},\dots,\textrm{g}_{a_k}\}$ be a subset of  $\mathfrak{g}$, i.e., $\textrm{g}_{a_s}\in \mathfrak{g}$ for all $s=1,\dots,k$ with $k\geq 1$. We introduce the following definition.
\begin{definition}\label{def2}
(i) The free zero pattern generated by $\mathfrak{g}_{a}$, denoted by $\Sigma_{\rm f}(\mathfrak{g}_{a})$, is defined as the generated linear subspace of $\mathfrak{g}_{a}$, i.e.,
$$
\Sigma_{\rm f}(\mathfrak{g}_{a}):=\Big\{ \sum_{s=1}^kl_s{\rm g}_{a_s}: l_s \in \mathbb{R},s=1,\dots,k\Big\}.
$$
(ii) The rigid  zero pattern generated by $\mathfrak{g}_{a}$, denoted by $\Sigma_{\rm r}(\mathfrak{g}_{a})$, is defined as the set of linear combinations  with nonzero coefficients of the elements in $\mathfrak{g}_{a}$, i.e.,
$$
\Sigma_{\rm r}(\mathfrak{g}_{a}):=\Big\{ \sum_{s=1}^kl_s{\rm g}_{a_s}: l_s\neq 0 \in \mathbb{R},s=1,\dots,k\Big\}.
$$
\end{definition}

\noindent{\bf Example 1.} The general linear group, ${\rm GL}(n)$, is the Lie group formed by all $\mathbb{R}^{n\times n}$ invertible matrices. It has two components separated by the set of singular matrices. The connected component of ${\rm GL}(n)$ containing $I_{n}$ is a Lie subgroup of ${\rm GL}(n)$, denoted by ${\rm GL^{+}}(n)$. The Lie algebra associated with Lie group ${\rm GL^{+}}(n)$ is $\mathfrak{gl}(n)$, the set of all $n\times n$ real matrices. Let $E_{ij}\in\R^{n\times n}$ be the matrix with $(i,j)$-th entry being $1$ and others being $0$.  Let $\mathfrak{g}$ be $\mathfrak{gl}(n)$.

\begin{figure}[H]
  \centering
  \includegraphics[width=7cm]{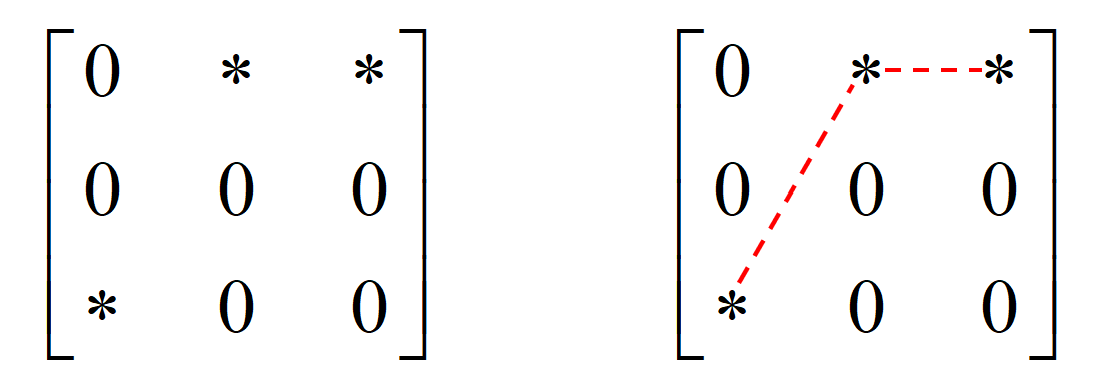}
  \caption{The form of the matrix in $\Sigma_{\rm r}(\mathfrak{g}_{a})$ (left) and the form of the matrix in $\Sigma_{\rm r}(\mathfrak{g}_{b})$ (right). The dashed lines indicate linear relations. }\label{fig1}
\end{figure}

(i) Let $\mathfrak{g}_{a}:=\{\textrm{g}_{a_1}, \textrm{g}_{a_2}, \textrm{g}_{a_3}\}$ with $\textrm{g}_{a_1}=E_{12}$, $\textrm{g}_{a_2}=E_{13}$, and $\textrm{g}_{a_3}=E_{31}$. Then an element $G\in \Sigma_{\rm f}(\mathfrak{g}_{a})$ may take value as  $E_{12}$ or $E_{13}$; an element $G\in \Sigma_{\rm r}(\mathfrak{g}_{a})$ can neither be $E_{12}$ nor $E_{13}$.

(ii) Let $\mathfrak{g}_{b}:=\{\textrm{g}_{b_1}\}$   with $\textrm{g}_{b_1}=E_{12}+2E_{13}+3E_{31}$. Then an element  $G\in \Sigma_{\rm f}(\mathfrak{g}_{b})$ or $G\in \Sigma_{\rm r}(\mathfrak{g}_{b})$ always satisfies that $G_{13}=2G_{12}$, $G_{31}=3G_{12}$; i.e., the zero pattern specified by $\mathfrak{g}_{b}$ imposes a linear equality constraints on the elements of the matrices therein. \hfill$\square$

\subsection{Structural Controllability}

Let   $\mathfrak{g}_{\alpha}:=\{\textrm{g}_{\alpha_1},\dots,\textrm{g}_{\alpha_d}\}$ and $\mathfrak{g}_{\beta}:=\{\textrm{g}_{\beta_1},\dots,\textrm{g}_{\beta_e}\}$ be two subsets of the Lie algebra $\mathfrak{g}$. We introduce the following definition on the structural controllability and accessibility  of the system \eqref{bilinear}.

\begin{definition}\label{def3}
The system \eqref{bilinear} is called structurally controllable (accessible) on the Lie group $\mathbf{G}$ with respect to the pair of zero patterns $\Sigma_{\rm zero}:=(\Sigma_{\rm r}(\mathfrak{g}_{\alpha}),\Sigma_{\rm f}(\mathfrak{g}_{\beta}))$ if there exist an integer $m\geq 1$, $\mathsf{B}_0\in \Sigma_{\rm r}(\mathfrak{g}_{\alpha})$ and $\mathsf{B}_i\in \Sigma_{\rm f}(\mathfrak{g}_{\beta}), i=1,\dots,m$, such that the system \eqref{bilinear} is controllable (accessible) on the Lie group $\mathbf{G}$.
\end{definition}

\section{Main Results}\label{sec:results}
\subsection{Structural Controllability over $\SO(n)$ }
The special orthogonal group,  $\SO(n)$, is  the Lie group formed by  all $\mathbb{R}^{n\times n}$ orthogonal matrices whose determinant is equal to one.   The Lie algebra of $\SO(n)$, $\so(n)$, consists of $n\times n$ real skew-symmetric matrices, which has the dimension $n(n-1)/2$. Define $B_{ij}=E_{ij}-E_{ji}$. Then the set $\mathpzc{B}=\{B_{ij}: 1\leq i < j \leq n\}$ forms a standard basis of $\so(n)$. Let $A_1,\dots,A_d$ be $d$ matrices in $\so(n)$. We introduce the pair of zero patterns  $\Sigma_{\rm zero}:=\big(\Sigma_{\rm r}(\so(n)_{\alpha}),\Sigma_{\rm f}(\so(n)_{\beta})\big)$ by
$$
\so(n)_{\alpha}:=\{A_{1},\dots,A_{d}\}, \quad \so(n)_{\beta}:=\{B_{i_1j_1},\dots,B_{i_ej_e}\},
$$
where  $B_{i_kj_k}\in\mathpzc{B}$ for $k=1,\dots,e$.

By the compactness of the Lie group $\SO(n)$, the system \eqref{bilinear} is controllable if and only if it is
accessible. Let $\mathrm{V}=\{1,2,\dots,n\}$. We first introduce the following graph representations of the zero patterns.
\begin{definition}\label{def4}
	$(i)$ The drift graph associated with $\Sigma_{\rm r}(\so(n)_{\alpha})$, denoted by ${\mathrm{G}}^{{}^\alpha}_{{\rm drift}}$, is  defined as the undirected graph $\mathrm{G}_{\rm drift}^{{}^\alpha}=\big(\mathrm{V},\mathrm{E}_{\rm drift}^{{}^\alpha}\big)$, where  $\{i,j\}\in\mathrm{E}_{\rm drift}^{{}^\alpha}$ if and only if there exists $s=1,\dots,d$ such that $[A_s]_{ij}=-[A_s]_{ji}\neq 0$. \\
	$(ii)$ The controlled   graph associated with $\Sigma_{\rm f}(\so(n)_{\beta})$, denoted by $\mathrm{G}_{\rm contr}^{{}^\beta}$, is  defined as the undirected graph $\mathrm{G}_{\rm contr}^{{}^\beta}=\big(\mathrm{V},\mathrm{E}_{\rm contr}^{{}^\beta}\big)$ with $\mathrm{E}_{\rm contr}^{{}^\beta}=\big\{\{i_1,j_1\},\dots, \{i_e,j_e\}\big\}$.
\end{definition}

A necessary condition for structural controllability of system \eqref{bilinear} is that the union graph $\mathrm{G}_{\rm drift}^{{}^\alpha} \mcup \mathrm{G}_{\rm contr}^{{}^\beta}$ must be connected  \cite{wang2020}. However, this condition is not sufficient. The following theorem gives a sufficient condition for system \eqref{bilinear} to be structurally controllable on the Lie group $\SO(n)$ with respect to the pair of zero patterns $\Sigma_{\rm zero}=\big(\Sigma_{\rm r}(\so(n)_{\alpha}),\Sigma_{\rm f}(\so(n)_{\beta})\big)$.

\begin{theorem}\label{thm1}
The system \eqref{bilinear} is structurally controllable on the Lie group $\SO(n)$ with respect to the pair of zero patterns $\Sigma_{\rm zero}=\big(\Sigma_{\rm r}(\so(n)_{\alpha}),\Sigma_{\rm f}(\so(n)_{\beta})\big)$ if each connected component of $\mathrm{G}_{\rm contr}^{{}^\beta}$ contains at least three nodes and the union graph $\mathrm{G}_{\rm drift}^{{}^\alpha} \mcup \mathrm{G}_{\rm contr}^{{}^\beta}$ is connected.
\end{theorem}

\noindent{\bf Example 2.}
Consider the system \eqref{bilinear} evolving on $\SO(6)$. Let $\so(6)_\alpha=\{A_1,A_2,A_3\}$ with $A_1=2B_{14}+B_{25}$, $A_2=B_{12}-B_{15}$, and $A_3=3B_{15}+2B_{25}$. Let $\so(6)_\beta=\{B_{12},B_{23},B_{13},B_{45},B_{56},B_{46}\}$. The drift graph associated with $\Sigma_{\rm r}(\so(6)_{\alpha})$ and the controlled graph associated with $\Sigma_{\rm f}(\so(6)_{\beta})$ are shown, respectively, in Figure \ref{fig2}.

\begin{figure}[H]
\centering
\begin{minipage}[c]{0.33\textwidth}
\centering
\includegraphics[height=4cm,width=4.7cm]{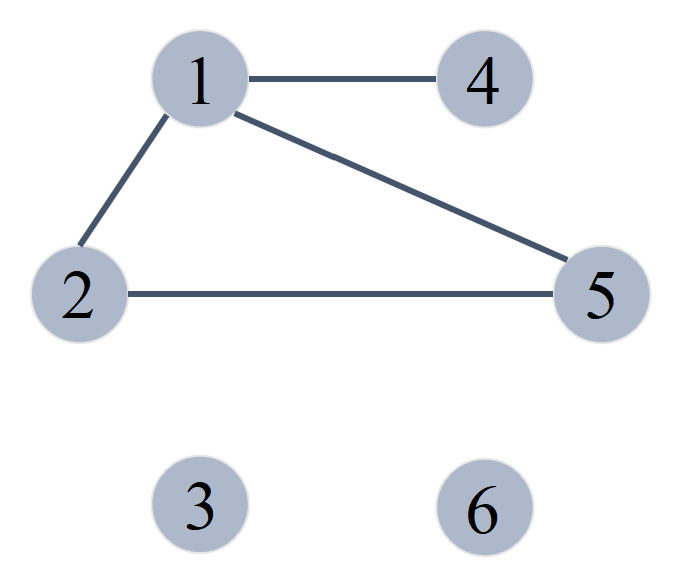}
\end{minipage}
\begin{minipage}[c]{0.33\textwidth}
\centering
\includegraphics[height=4cm,width=4.8cm]{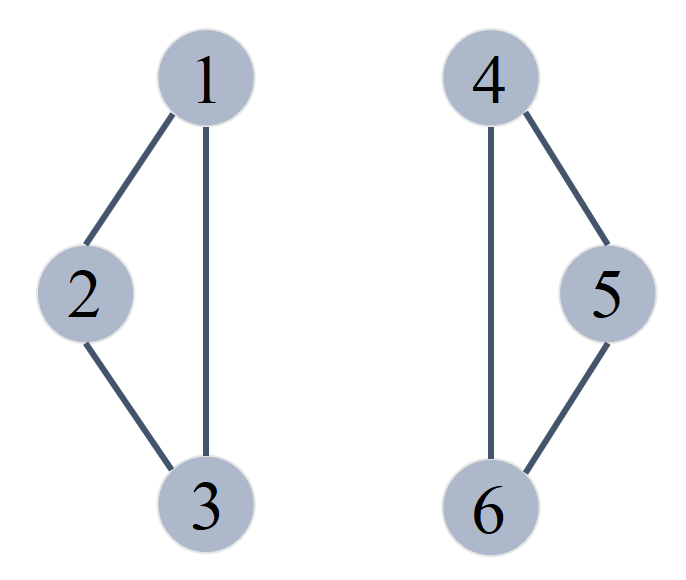}
\end{minipage}
\caption{The graph $\mathrm{G}_{\rm drift}^{{}^\alpha}$ (left) and the graph $\mathrm{G}_{\rm contr}^{{}^\beta}$ (right).}\label{fig2}
\end{figure}
It is clear that each connected component of $\mathrm{G}_{\rm contr}^{{}^\beta}$ contains at least three nodes, and the union graph $\mathrm{G}_{\rm drift}^{{}^\alpha} \mcup \mathrm{G}_{\rm contr}^{{}^\beta}$ is connected. As a result, the graphical condition of Theorem \ref{thm1} has been met. Choose $A=A_1+3A_2+A_3\in\Sigma_{\rm r}(\so(6)_{\alpha})$, and by direct computation one can verify
\begin{equation}\label{ex.2}
\{A,B_{12},B_{23},B_{45},B_{56}\}_{\rm LA}=\so(6).
\end{equation}
This implies that the system \eqref{bilinear} is structurally controllable on the Lie group $\SO(6)$ with respect to the pair of zero patterns $\Sigma_{\rm zero}=\big(\Sigma_{\rm r}(\so(6)_{\alpha}),\Sigma_{\rm f}(\so(6)_{\beta})\big)$, providing a validation of Theorem \ref{thm1}. \hfill$\square$

\subsection{Structural Accessibility over ${\rm GL}^{+}(n)$ }
Recall that the Lie algebra of ${\rm GL^{+}}(n)$ is $\gl(n)$, the set of all $n\times n$ real matrices. The set $\mathpzc{E}=\{E_{ij}: 1\leq i,j\leq n\}$ forms a basis
of $\gl(n)$. Let $A_1,\dots,A_d$ be $d$ matrices in $\gl(n)$. A pair of zero patterns  $\Sigma_{\rm zero}:=\big(\Sigma_{\rm r}(\gl(n)_{\alpha}),\Sigma_{\rm f}(\gl(n)_{\beta})\big)$ is given by
$$
\gl(n)_{\alpha}:=\{A_{1},\dots,A_{d}\}, \quad \gl(n)_{\beta}:=\{E_{i_1j_1},\dots,E_{i_ej_e}\}.
$$
We similarly introduce the following graph representations of these zero patterns, which now have to be directed graphs.
\begin{definition}\label{def5}
	$(i)$ The drift graph associated with $\Sigma_{\rm r}(\gl(n)_{\alpha})$, denoted by $\mathcal{G}_{\rm drift}^{{}^\alpha}$, is  defined as the directed graph $\mathcal{G}_{\rm drift}^{{}^\alpha}=\big(\mathrm{V},\mathcal{E}_{\rm drift}^{{}^\alpha}\big)$, where  $(i,j)\in \mathcal{E}_{\rm drift}^{{}^\alpha}$ if and only if there exists $s=1,\dots,d$ such that $[A_s]_{ij}\neq 0$. \\
	$(ii)$ The controlled   graph associated with $\Sigma_{\rm f}(\gl(n)_{\beta})$, denoted by $\mathcal{G}_{\rm contr}^{{}^\beta}$, is  defined as the directed graph $\mathcal{G}_{\rm contr}^{{}^\beta}=\big(\mathrm{V},\mathcal{E}_{\rm contr}^{{}^\beta}\big)$ with $\mathcal{E}_{\rm contr}^{{}^\beta}=\big\{(i_1,j_1),\dots, (i_e,j_e)\big\}$.
\end{definition}
Note that Lie group ${\rm GL^{+}}(n)$ is connected, but not compact. A necessary condition for the system \eqref{bilinear} to be controllable is that it is accessible. In the presence of the drift term $\mathsf{B}_0$, accessibility
(plus the connectedness of ${\rm GL^{+}}(n)$) is not sufficient for controllability. In this
case we give some sufficient conditions, and a necessary condition for the system \eqref{bilinear} to be structurally accessible on the Lie group ${\rm GL^{+}}(n)$ with respect to the pair of zero patterns $\Sigma_{\rm zero}=\big(\Sigma_{\rm r}(\gl(n)_{\alpha}),\Sigma_{\rm f}(\gl(n)_{\beta})\big)$.

When the union graph $\mathcal{G}_{\rm drift}^{{}^\alpha} \mcup \mathcal{G}_{\rm contr}^{{}^\beta}$ is not strongly connected, the system \eqref{bilinear} cannot be structurally accessible on the Lie group ${\rm GL^{+}}(n)$ with respect to the pair of zero patterns $\Sigma_{\rm zero}=\big(\Sigma_{\rm r}(\gl(n)_{\alpha}),\Sigma_{\rm f}(\gl(n)_{\beta})\big)$ \cite{wang2020}. The following theorem establishes a sufficient condition under which the system \eqref{bilinear} is structurally accessible.

\begin{theorem}\label{thm2.1}
 The system \eqref{bilinear} is structurally accessible on the Lie group ${\rm GL^{+}}(n)$ with respect to the pair of zero patterns $\Sigma_{\rm zero}=\big(\Sigma_{\rm r}(\gl(n)_{\alpha}),\Sigma_{\rm f}(\gl(n)_{\beta})\big)$ if the following conditions hold:

 $(i)$ Each weakly connected component of $\mathcal{G}_{\rm contr}^{{}^\beta}$ is strongly connected with at least two nodes;

 $(ii)$ The union graph $\mathcal{G}_{\rm drift}^{{}^\alpha} \mcup \mathcal{G}_{\rm contr}^{{}^\beta}$ is strongly connected;

 $(iii)$ $\mathcal{G}_{\rm contr}^{{}^\beta}$ has at least one self-loop.

\end{theorem}

In particular, when $\gl(n)_{\alpha}$ is a subset of $\mathpzc{E}$, the condition that $\mathcal{G}_{\rm contr}^{{}^\beta}$ has at least one self-loop can be relaxed.

\begin{theorem}\label{thm2.2}
 Suppose $\gl(n)_{\alpha}\subseteq\mathpzc{E}$. Then the system \eqref{bilinear} is structurally accessible on the Lie group ${\rm GL^{+}}(n)$ with respect to the pair of zero patterns $\Sigma_{\rm zero}=\big(\Sigma_{\rm r}(\gl(n)_{\alpha}),\Sigma_{\rm f}(\gl(n)_{\beta})\big)$ if each weakly connected component of $\mathcal{G}_{\rm contr}^{{}^\beta}$ is strongly connected with at least two nodes, and the union graph $\mathcal{G}_{\rm drift}^{{}^\alpha}\mcup \mathcal{G}_{\rm contr}^{{}^\beta}$ is a strongly connected digraph with at least one self-loop.
\end{theorem}

\noindent{\bf Example 3.}
Consider the system \eqref{bilinear} evolving on ${\rm GL^{+}}(4)$. Let $\gl(4)_\alpha=\{A_1,A_2,A_3\}$ with $A_1=3E_{13}+E_{42}$, $A_2=E_{12}-E_{13}$, and $A_3=E_{33}-E_{44}+2E_{31}$. Let $\gl(4)_\beta=\{E_{11},E_{12},E_{21},E_{34},E_{43}\}$. The drift graph associated with $\Sigma_{\rm r}(\gl(4)_{\alpha})$ and the controlled graph associated with $\Sigma_{\rm f}(\gl(4)_{\beta})$ are shown, respectively, in Figure \ref{fig3}.
\begin{figure}[H]
\centering
\begin{minipage}[c]{0.33\textwidth}
\centering
\includegraphics[height=3.4cm,width=4.2cm]{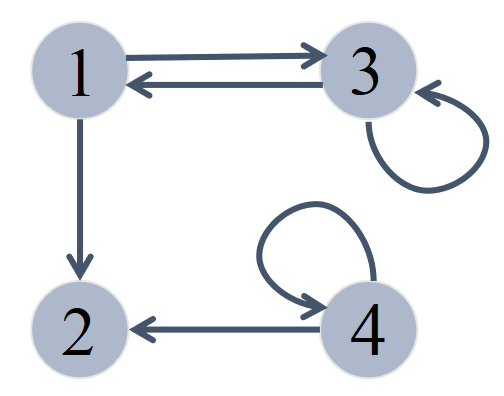}
\end{minipage}
\begin{minipage}[c]{0.33\textwidth}
\centering
\includegraphics[height=3.4cm,width=3.9cm]{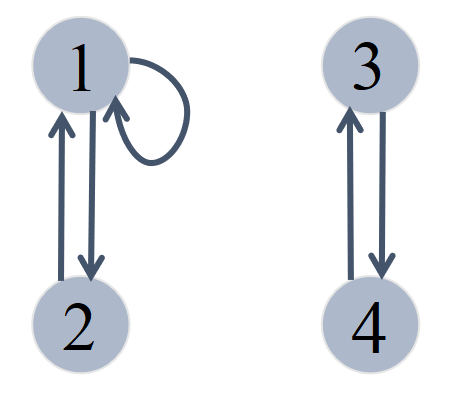}
\end{minipage}
\caption{The graph $\mathcal{G}_{\rm drift}^{{}^\alpha}$ (left) and the graph $\mathcal{G}_{\rm contr}^{{}^\beta}$ (right).}\label{fig3}
\end{figure}
 (i) It is easily seen that each weakly connected component of $\mathcal{G}_{\rm contr}^{{}^\beta}$ is strongly connected with two nodes. In addition, $\mathcal{G}_{\rm contr}^{{}^\beta}$ has one self-loop and the union graph $\mathcal{G}_{\rm drift}^{{}^\alpha} \mcup \mathcal{G}_{\rm contr}^{{}^\beta}$ is  strongly connected. Hence, the graphical condition of Theorem \ref{thm2.1} has been met. Choose $A=A_1+2A_2+A_3\in\Sigma_{\rm r}(\gl(4)_{\alpha})$. By direct computation one can verify $\{A,E_{11},E_{12},E_{21},E_{34},E_{43}\}_{\rm LA}=\gl(4)$. Thus, the system \eqref{bilinear} is structurally accessible on the Lie group ${\rm GL^{+}}(4)$ with respect to the pair of zero patterns $\Sigma_{\rm zero}=\big(\Sigma_{\rm r}(\gl(4)_{\alpha}),\Sigma_{\rm f}(\gl(4)_{\beta})\big)$. This provides a validation of Theorem \ref{thm2.1}.

(ii) Let $\gl(4)_\beta=\{E_{12},E_{21},E_{34},E_{43}\}$. Then $\mathcal{G}_{\rm contr}^{{}^\beta}$ has no self-loops. Note that $\mathcal{G}_{\rm drift}^{{}^\alpha}$ has two self-loops. A direct verification shows that
$\{A,E_{12},E_{21},E_{34},E_{43}\}_{\rm LA}=\mathfrak{sl}(4)$ for any $A\in\Sigma_{\rm r}(\gl(4)_{\alpha})$. This is to say, for any integer $m\geq 1$, and for any $\mathsf{B}_0\in \Sigma_{\rm r}(\gl(4)_{\alpha})$, $\mathsf{B}_i\in \Sigma_{\rm f}(\gl(4)_{\beta})$, we have $\{\mathsf{B}_0,\mathsf{B}_1,\dots,\mathsf{B}_m\}_{\rm LA}\neq \gl(4)$. Therefore, the system \eqref{bilinear} is not structurally accessible on the Lie group ${\rm GL^{+}}(4)$ with respect to the pair of zero patterns $\Sigma_{\rm zero}=\big(\Sigma_{\rm r}(\gl(4)_{\alpha}),\Sigma_{\rm f}(\gl(4)_{\beta})\big)$. Generally speaking, the self-loop of $\mathcal{G}_{\rm drift}^{{}^\alpha}$ cannot replace the role of the self-loop of $\mathcal{G}_{\rm contr}^{{}^\beta}$, which is why we require the set $\gl(n)_{\alpha}$ to be a subset of $\mathpzc{E}$ in Theorem \ref{thm2.2}. \hfill$\square$

\medskip

\noindent{\bf Example 4.} Consider the system \eqref{bilinear} evolving on ${\rm GL^{+}}(4)$. Let $\gl(4)_\alpha=\{E_{12},E_{13},E_{31},E_{33},E_{42}\}$, $\gl(4)_\beta=\{E_{12},E_{21},E_{34},E_{43}\}$. The drift graph associated with $\Sigma_{\rm r}(\gl(4)_{\alpha})$ and the controlled graph associated with $\Sigma_{\rm f}(\gl(4)_{\beta})$ are shown, respectively, in Figure \ref{fig4}.
\begin{figure}[H]
\centering
\begin{minipage}[c]{0.33\textwidth}
\centering
\includegraphics[height=3.4cm,width=4.2cm]{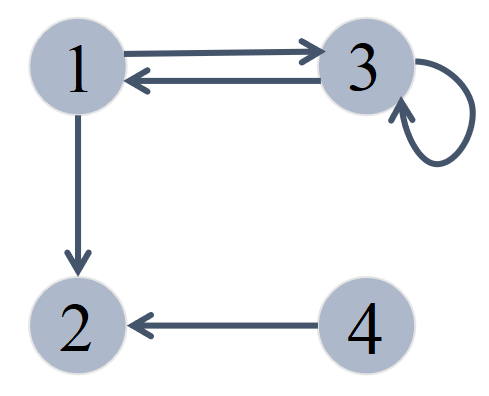}
\end{minipage}
\begin{minipage}[c]{0.33\textwidth}
\centering
\includegraphics[height=3.4cm,width=3.8cm]{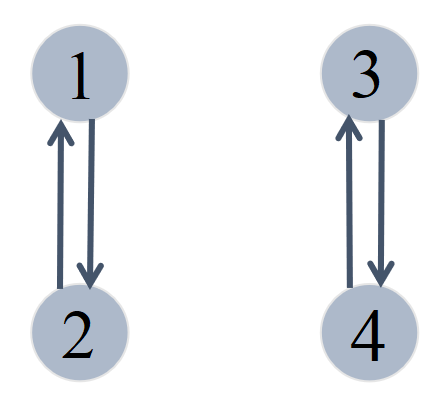}
\end{minipage}
\caption{The graph $\mathcal{G}_{\rm drift}^{{}^\alpha}$ (left) and the graph $\mathcal{G}_{\rm contr}^{{}^\beta}$ (right).}\label{fig4}
\end{figure}
Each weakly connected component of $\mathcal{G}_{\rm contr}^{{}^\beta}$ is strongly connected with two nodes. The union graph $\mathcal{G}_{\rm drift}^{{}^\alpha} \mcup \mathcal{G}_{\rm contr}^{{}^\beta}$ continues to be strongly connected. Moreover, $\mathcal{G}_{\rm drift}^{{}^\alpha}$ has a self-loop. Therefore, the graphical condition of Theorem \ref{thm2.2} is satisfied. For any $A\in\Sigma_{\rm r}(\gl(4)_{\alpha})$, we have $A=l_1E_{12}+l_2E_{13}+l_3E_{31}+l_4E_{33}+l_5E_{42}$ with $l_s\neq0\in\mathbb{R}$ for $s=1,\dots,5$. Write $\widetilde{A}=A-l_1E_{12}$.  Then $A':=\widetilde{A}-[[\widetilde{A},E_{21}],E_{12}]=l_3E_{31}+l_4E_{33}$, and $A'-[[A',E_{12}],E_{21}]=l_4E_{33}$. We thus get $E_{33}\in\{A,E_{12},E_{21},E_{34},E_{43}\}_{\rm LA}$. This allows us to further obtain $\{A,E_{12},E_{21},E_{34},E_{43}\}_{\rm LA}=\gl(4)$. As a result, the system \eqref{bilinear} is structurally accessible on the Lie group ${\rm GL^{+}}(4)$ with respect to the pair of zero patterns $\Sigma_{\rm zero}=\big(\Sigma_{\rm r}(\gl(4)_{\alpha}),\Sigma_{\rm f}(\gl(4)_{\beta})\big)$. This provides a validation of Theorem \ref{thm2.2}. \hfill$\square$

\subsection{Structural Controllability over ${\rm SU}(n)$ }
The special unitary group, ${\rm SU}(n)$, is the Lie group formed by all $\mathbb{C}^{n\times n}$ unitary matrices with determinant one. The Lie algebra of ${\rm SU}(n)$ is equal
to $\su(n)$, the set of all $\mathbb{C}^{n\times n}$ skew-Hermitian matrices with trace zero. Lie algebra $\su(n)$ is a finite dimensional vector space over $\mathbb{R}$. Let $\sym(n)$ denote the vector space containing all real traceless symmetric matrices. The Cartan-like decomposition of $\su(n)$ is given by $\su(n)= \so(n)\oplus {\rm i}\sym(n)$,  where ${\rm i}$ is the imaginary unit. It follows easily that the dimension of $\su(n)$ is equal to $n^2-1$. Let $C_{ij}={\rm i}(E_{ij}+E_{ji})$, $D_{ij}={\rm i}(E_{ii}-E_{jj})$. Define $\mathpzc{C}=\{C_{ij}: 1\leq i< j\leq n\}$, $\mathpzc{D} = \{D_{ij}: 1\leq i< j\leq n\}$. The maximal linearly independent set of $\mathpzc{D}$, denoted by $\mathpzc{D}_L$, contains $n-1$ elements. Then $\mathpzc{B}\mcup\mathpzc{C}\mcup\mathpzc{D}_L$ forms a basis of $\su(n)$.

Let $A_1,\dots,A_d$ be $d$ matrices in $\su(n)$. Let ${\rm S}_1\subseteq \mathpzc{B}$, ${\rm S}_2\subseteq \mathpzc{C}$, ${\rm S}_3\subseteq \mathpzc{D}$ be three subsets of $\su(n)$. We introduce
$$
\su(n)_{\alpha}:=\{A_{1},\dots,A_{d}\}, \quad \su(n)_{\beta}:={\rm S}_1\mcup{\rm S}_2\mcup{\rm S}_3,
$$
so a pair of zero patterns $\Sigma_{\rm zero}:=\big(\Sigma_{\rm r}(\su(n)_{\alpha}),\Sigma_{\rm f}(\su(n)_{\beta})\big)$ is induced.
Set $\textbf{C}=\{{\rm Blue, Red, Green}\}$. The graphs corresponding to the zero patterns $\Sigma_{\rm r}(\su(n)_{\alpha})$ and $\Sigma_{\rm f}(\su(n)_{\beta})$ are edge-colored multigraphs.
\begin{definition}\label{def6}
	$(i)$ The drift graph associated with $\Sigma_{\rm r}(\su(n)_{\alpha})$, denoted by $\mathscr{G}_{\rm drift}^{{}^\alpha}$, is  defined as the edge-colored multigraph $\mathscr{G}_{\rm drift}^{{}^\alpha}=\big(\mathrm{V},\mathscr{E}_{\rm drift}^{{}^\alpha}\big)$ with $\mathscr{E}_{\rm drift}^{{}^\alpha}=\mathscr{E}_{\rm Blue}\mcup\mathscr{E}_{\rm Red}\mcup\mathscr{E}_{\rm Green}$, where  $\{i,j;{\rm Blue}\}\in \mathscr{E}_{\rm Blue}$ if and only if there exists $s=1,\dots,d$ such that ${\rm Re}([A_s]_{ij})=-{\rm Re}([A_s]_{ji})\neq 0$; $\{i,j;{\rm Red}\}\in \mathscr{E}_{\rm Red}$ if and only if there exists $s=1,\dots,d$ such that ${\rm Im}([A_s]_{ij})={\rm Im}([A_s]_{ji})\neq 0$; and $\{i,i;{\rm Green}\}\in \mathscr{E}_{\rm Green}$ if and only if there exists $s=1,\dots,d$ such that $[A_s]_{ii}\neq 0$. \\
	$(ii)$ The controlled graph associated with $\Sigma_{\rm f}(\su(n)_{\beta})$, denoted by $\mathscr{G}_{\rm contr}^{{}^\beta}$, is  defined as the edge-colored multigraph $\mathscr{G}_{\rm contr}^{{}^\beta}=\big(\mathrm{V},\mathscr{E}_{\rm contr}^{{}^\beta}\big)$ with $\mathscr{E}_{\rm contr}^{{}^\beta}=\mathscr{E}_{{\rm S}_1}\mcup\mathscr{E}_{{\rm S}_2}\mcup\mathscr{E}_{{\rm S}_3}$, where  $\mathscr{E}_{{\rm S}_1}=\big\{\{i,j;{\rm Blue}\}: B_{ij}\in{\rm S}_1\big\}, \mathscr{E}_{{\rm S}_2}=\big\{\{i,j;{\rm Red}\}: C_{ij}\in{\rm S}_2\big\}$, and $\mathscr{E}_{{\rm S}_3}=\big\{\{k,k;{\rm Green}\}: k=i,j, D_{ij}\in{\rm S}_3\big\}$.
\end{definition}

Similar to ${\rm SO}(n)$, ${\rm SU}(n)$ is also connected and compact, the controllability analysis of the system \eqref{bilinear}  on  ${\rm SU}(n)$ relies on the investigation of the equivalence
between the system Lie algebra and the underlying Lie algebra $\su(n)$. In this section we shall look for necessary and sufficient conditions for the system \eqref{bilinear} to be structurally controllable on the Lie group ${\rm SU}(n)$ with respect to the pair of zero patterns $\Sigma_{\rm zero}=\big(\Sigma_{\rm r}(\su(n)_{\alpha}),\Sigma_{\rm f}(\su(n)_{\beta})\big)$.

Note that, for the graph $\mathscr{G}_{\rm contr}^{{}^\beta}$, there are two cases (a) $\mathscr{G}_{\rm contr}^{{}^\beta}$ is connected; (b) $\mathscr{G}_{\rm contr}^{{}^\beta}$ is not connected. First of all, when $\mathscr{G}_{\rm contr}^{{}^\beta}$ is connected, we have the following theorem.

\begin{theorem}\label{thm4}
Suppose $\mathscr{G}_{\rm contr}^{{}^\beta}$ is connected. Then the system \eqref{bilinear} is structurally controllable on the Lie group ${\rm SU}(n)$ with respect to the pair of zero patterns $\Sigma_{\rm zero}=\big(\Sigma_{\rm r}(\su(n)_{\alpha}),\Sigma_{\rm f}(\su(n)_{\beta})\big)$ if and only if the union graph $\mathscr{G}_{\rm drift}^{{}^\alpha}\mcup\mathscr{G}_{\rm contr}^{{}^\beta}$ has a self-loop or a cycle with an odd number of Red edges.
\end{theorem}

It is worth pointing out that a multi-edge $\big\{\{i,j;{\rm Blue}\},\{i,j;{\rm Red}\}\big\}$ in $\mathscr{G}_{\rm drift}^{{}^\alpha}\mcup\mathscr{G}_{\rm contr}^{{}^\beta}$ can form a cycle with an odd number of Red edges. We now turn to the case when $\mathscr{G}_{\rm contr}^{{}^\beta}$ is not connected. An obvious necessary condition for structural controllability is that the union graph $\mathscr{G}_{\rm drift}^{{}^\alpha}\mcup\mathscr{G}_{\rm contr}^{{}^\beta}$ must be connected. This condition is however not sufficient.

\begin{theorem}\label{thm5}
Suppose $\mathscr{G}_{\rm contr}^{{}^\beta}$ is not connected. Then the system \eqref{bilinear} is structurally controllable on the Lie group ${\rm SU}(n)$ with respect to the pair of zero patterns $\Sigma_{\rm zero}=\big(\Sigma_{\rm r}(\su(n)_{\alpha}),\Sigma_{\rm f}(\su(n)_{\beta})\big)$ if the following conditions hold:

$(i)$ Each connected component of $\mathscr{G}_{\rm contr}^{{}^\beta}$ contains at least three nodes;

$(ii)$  $\mathscr{G}_{\rm drift}^{{}^\alpha}$ has no multi-edges and the union graph $\mathscr{G}_{\rm drift}^{{}^\alpha}\mcup\mathscr{G}_{\rm contr}^{{}^\beta}$ is connected;

$(iii)$ $\mathscr{G}_{\rm drift}^{{}^\alpha}\mcup\mathscr{G}_{\rm contr}^{{}^\beta}$ has a self-loop or a cycle with an odd number of Red edges.
\end{theorem}

\noindent{\bf Example 5.} Consider the system \eqref{bilinear} evolving on ${\rm SU}(5)$. Let $\su(5)_\alpha=\{A_{1},A_{2},A_{3}\}$ with $A_1=C_{12}-2B_{25}$, $A_2=2B_{12}+D_{15}+3C_{45}$, and $A_3=B_{34}-C_{34}$. Let $\su(5)_\beta=\{B_{12},C_{13},C_{34},B_{14},B_{15},C_{15},D_{24}\}$. The drift graph associated with $\Sigma_{\rm r}(\su(5)_{\alpha})$ and the controlled graph associated with $\Sigma_{\rm f}(\su(5)_{\beta})$ are shown, respectively, in Figure \ref{fig5}.
\begin{figure}[H]
\centering
\begin{minipage}[c]{0.33\textwidth}
\centering
\includegraphics[height=4cm,width=5cm]{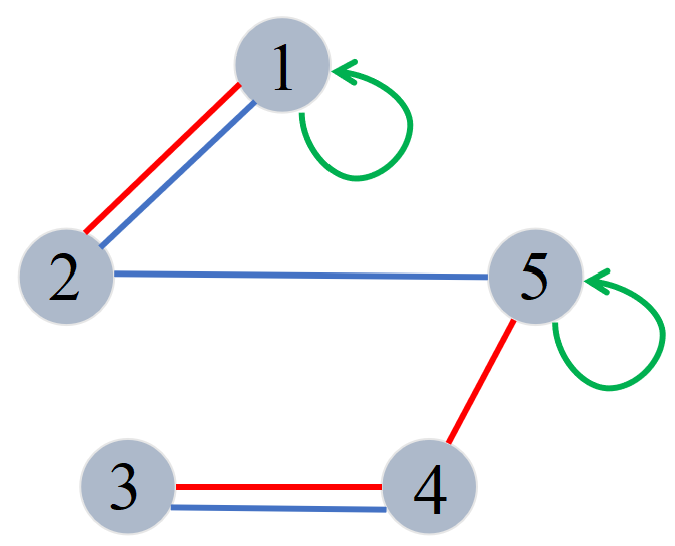}
\end{minipage}
\begin{minipage}[c]{0.33\textwidth}
\centering
\includegraphics[height=4cm,width=4.5cm]{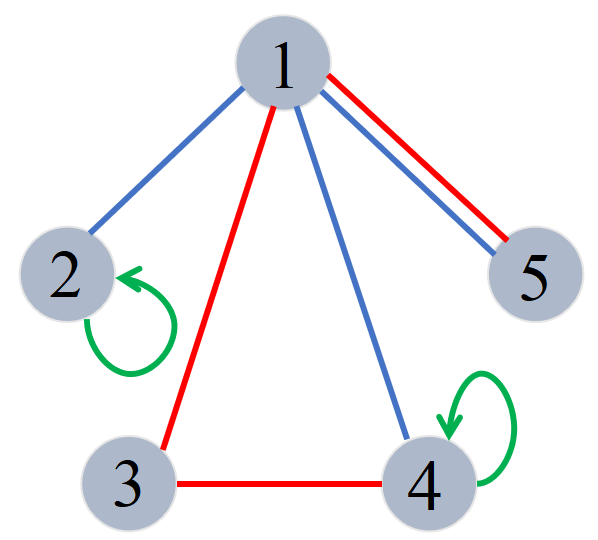}
\end{minipage}
\caption{The  edge-colored multigraph $\mathscr{G}_{\rm drift}^{{}^\alpha}$ (left) and the edge-colored
multigraph $\mathscr{G}_{\rm contr}^{{}^\beta}$ (right).}\label{fig5}
\end{figure}

(i) It is evident that $\mathscr{G}_{\rm contr}^{{}^\beta}$ is connected. Note that $\mathscr{G}_{\rm contr}^{{}^\beta}$ has two self-loops. An easy computation shows that $\{B_{12},C_{13},B_{14},B_{15},D_{24}\}_{\rm LA}=\su(5)$. On the other hand, $\mathscr{G}_{\rm contr}^{{}^\beta}$ has a cycle containing an odd number of Red edges. This gives $\{B_{12},C_{13},B_{14},B_{15},C_{15}\}_{\rm LA}=\su(5)$. Both of these two situations will lead to structural controllability of the system \eqref{bilinear} on the Lie group ${\rm SU}(5)$ with respect to the pair of zero patterns $\Sigma_{\rm zero}=\big(\Sigma_{\rm r}(\su(5)_{\alpha}),\Sigma_{\rm f}(\su(5)_{\beta})\big)$.

(ii) Let $\su(5)_\beta=\{B_{12},C_{13},C_{34},B_{14},B_{15}\}$. Then $\mathscr{G}_{\rm contr}^{{}^\beta}$ is still connected, but it has no self-loops, and no cycles with an odd number of Red edges. We have $\{B_{12},C_{13},C_{34},B_{14},B_{15}\}_{\rm LA}\neq\su(5)$, which is easy to check.  However, $\mathscr{G}_{\rm drift}^{{}^\alpha}$ has two self-loops. Choose $A=A_1+A_2+A_3\in\Sigma_{\rm r}(\su(5)_{\alpha})$. A trivial verification shows that $A+2[B_{15},B_{12}]-2B_{12}+C_{34}=C_{12}+D_{15}+3C_{45}+B_{34}$.
Write $\widetilde{A}=C_{12}+D_{15}+3C_{45}+B_{34}$, and we have $[[\widetilde{A},C_{34}],B_{14}]=2C_{14}$. This implies that $C_{14}\in\{A,B_{12},C_{13},C_{34},B_{14},B_{15}\}_{\rm LA}$. By direct computation one can verify $\big\{\{C_{14}\}\mcup\{B_{12},C_{13},C_{34},B_{14},B_{15}\}\big\}_{\rm LA}=\su(5)$. We thus get $\{A,B_{12},C_{13},C_{34},B_{14},B_{15}\}_{\rm LA}=\so(5)$, and, in consequence, the system \eqref{bilinear} is structurally controllable on the Lie group ${\rm SU}(5)$ with respect to the pair of zero patterns $\Sigma_{\rm zero}=\big(\Sigma_{\rm r}(\su(5)_{\alpha}),\Sigma_{\rm f}(\su(5)_{\beta})\big)$.

If we replace $A_2=2B_{12}+D_{15}+3B_{45}$ with $A_2=2B_{12}+3B_{45}$, then $\mathscr{G}_{\rm drift}^{{}^\alpha}$ has no self-loops. However, $\mathscr{G}_{\rm drift}^{{}^\alpha}\mcup\mathscr{G}_{\rm contr}^{{}^\beta}$ has at least one cycle with an odd number of Red edges. For $A=A_1+A_2+A_3\in\Sigma_{\rm r}(\su(5)_{\alpha})$, we have $A+2[B_{15},B_{12}]-2B_{12}+C_{34}=C_{12}+3C_{45}+B_{34}$ and $[B_{12},C_{12}+3C_{45}+B_{34}]=2D_{12}$. Now $D_{12}\in\{A,B_{12},C_{13},C_{34},B_{14},B_{15}\}_{\rm LA}$, and so $\{A,B_{12},C_{13},C_{34},B_{14},B_{15}\}_{\rm LA}=\su(5)$. Therefore, the system \eqref{bilinear} is also structurally controllable on the Lie group ${\rm SU}(5)$ with respect to the pair of zero patterns $\Sigma_{\rm zero}=\big(\Sigma_{\rm r}(\su(5)_{\alpha}),\Sigma_{\rm f}(\su(5)_{\beta})\big)$.
This example provides a validation of Theorem \ref{thm4}. \hfill$\square$

\medskip

\noindent{\bf Example 6.} Consider the system \eqref{bilinear} evolving on ${\rm SU}(6)$. Let $\su(6)_\alpha=\{A_{1}, A_{2}, A_{3}\}$ with $A_1=C_{14}+2B_{45}$, $A_2=3B_{15}-2C_{25}+D_{25}$, and $A_3=B_{56}-C_{36}$. Let $\su(6)_\beta=\{B_{12},B_{13},C_{13},C_{23},B_{46},C_{56},B_{56},D_{45}\}$. The drift graph associated with $\Sigma_{\rm r}(\su(6)_{\alpha})$ and the controlled graph associated with $\Sigma_{\rm f}(\su(6)_{\beta})$ are shown, respectively, in Figure \ref{fig6}.
\begin{figure}[H]
\centering
\begin{minipage}[c]{0.33\textwidth}
\centering
\includegraphics[height=4cm,width=5.2cm]{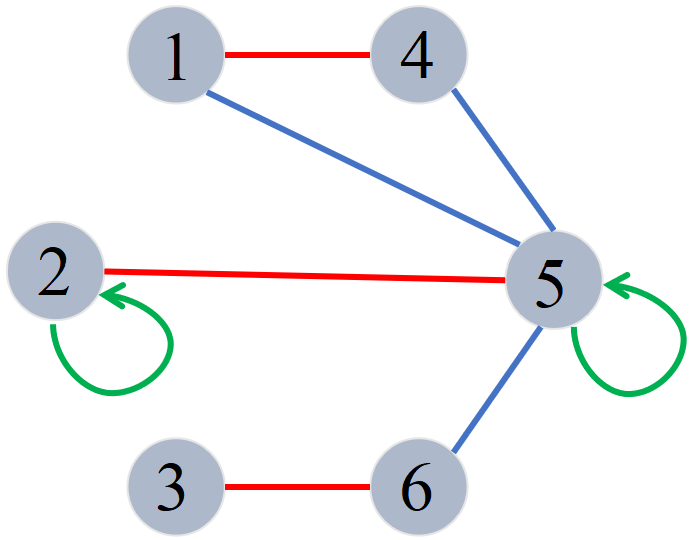}
\end{minipage}
\begin{minipage}[c]{0.33\textwidth}
\centering
\includegraphics[height=4cm,width=5.2cm]{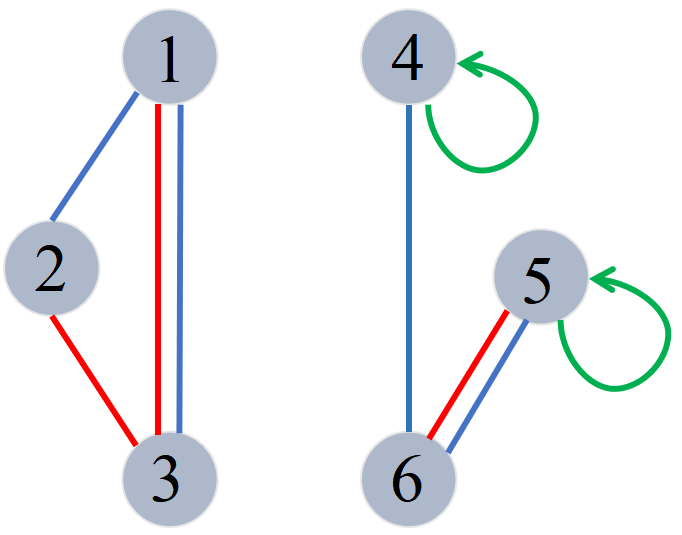}
\end{minipage}
\caption{The edge-colored multigraph $\mathscr{G}_{\rm drift}^{{}^\alpha}$ (left) and the edge-colored multigraph $\mathscr{G}_{\rm contr}^{{}^\beta}$ (right).}\label{fig6}
\end{figure}
Obviously, $\mathscr{G}_{\rm drift}^{{}^\alpha}$ has no multiple edges. Each connected component of $\mathscr{G}_{\rm contr}^{{}^\beta}$ contains at least three nodes and the union graph $\mathscr{G}_{\rm drift}^{{}^\alpha}\mcup \mathscr{G}_{\rm contr}^{{}^\beta}$ is connected. First, note that $\mathscr{G}_{\rm contr}^{{}^\beta}$ has two self-loops.  For any $A\in\Sigma_{\rm r}(\su(6)_{\alpha})$, we have $A=l_1A_{1}+l_2A_{2}+l_3A_{3}$ with $l_s\neq0\in\mathbb{R}$ for $s=1,2,3$. An easy computation shows that
\begin{equation}\label{ex.1}
[[A,B_{13}],B_{12}]=l_3C_{26}.
\end{equation}
This gives $C_{26}\in\{A,B_{12},B_{13},B_{46},C_{56},D_{45}\}_{\rm LA}$, and we thus get $\{A,B_{12},B_{13},B_{46},C_{56},D_{45}\}_{\rm LA}=\su(6)$. Hence, the system \eqref{bilinear} is structurally controllable on the Lie group ${\rm SU}(6)$ with respect to the pair of zero patterns $\Sigma_{\rm zero}=\big(\Sigma_{\rm r}(\su(6)_{\alpha}),\Sigma_{\rm f}(\su(6)_{\beta})\big)$.

Similarly it can be observed that $\mathscr{G}_{\rm contr}^{{}^\beta}$ has a cycle with an odd number of Red edges. By \eqref{ex.1} it is immediate that $C_{26}\in\{A,B_{12},B_{13},C_{23},B_{46},C_{56}\}_{\rm LA}$. This clearly forces $\{A,B_{12},B_{13},C_{23},B_{46},C_{56}\}_{\rm LA}=\su(6)$. Thus, one can deduce the structural controllability of system \eqref{bilinear}. This example provides a validation of Theorem \ref{thm5}. \hfill$\square$

\section{Conclusions}\label{sec:conc}

We have established  graphical conditions for structural controllability and accessibility  of drifted bilinear systems over Lie groups.  Bilinear control systems with   drift  and  controlled terms that evolve over Lie groups were considered. Zero patterns   prescribed   the possible configuration of the drift and controlled dynamics with respect to a set of base elements in the corresponding Lie algebra. The drift dynamics was assumed to be fixed, taking values as a linear combination of the base elements with strictly non-zero coefficients; the controlled dynamics was assumed to be free with potentially zero coefficients in the configuration. For bilinear systems over  the special orthogonal group or the special unitary group, the zero patterns were proven  to be associated with two undirected or directed graphs whose connectivity and connected components  ensure structural controllability/accessibility. For bilinear systems over the special unitary group,   two edge-colored graphs associated with the drift and controlled zero patterns were proven to be critical for investigating structural controllability. Future work might include  necessary and sufficient  graphical conditions for structural controllability and accessibility, and a computational complexity analysis of the established connectivity conditions. 

\section*{Appendix}
\subsection*{A. Proof of Theorem \ref{thm1}}
We first recall a few auxiliary lemmas that are used in the proof. Consider  $\mathfrak{G}$ as the set of all undirected graphs over the node set $\mathrm{V}$, and $\mathfrak{B}$ as the class of all subsets of $\mathpzc{B}$. By identifying each edge $\{i,j\}$ with the matrix $B_{ij}$, we establish a natural $1-1$ correspondence, between each element in $\mathfrak{G}$ and each element in $\mathfrak{B}$. We denote such a mapping $\ell$ which maps from $\mathfrak{B}$ to $\mathfrak{G}$.

\begin{lemma}\label{so.1}
Let $\mathcal {S}$ be a subset of $\mathpzc{B}$ and $\ell(\mathcal {S})$ be the graph associated
with $\mathcal {S}$. The Lie algebra generated by $\mathcal {S}$ is equal to $\so(n)$ if and and if $\ell(\mathcal {S})$ is connected.
\end{lemma}
For any $A\in\so(n)$, we can represent it uniquely in the form 
$$
A=\sum_{k=1}^{l} a_{k} B_{{i}^\ast_k{j}^\ast_k},
$$
where $a_k\neq 0\in\mathbb{R}$, $i^\ast_k,j^\ast_k\in\mathrm{V}$, and $B_{{i}^\ast_k{j}^\ast_k}\in\mathpzc{B}$. Define $\mathrm{G}_A=(\mathrm{V},\mathrm{E}_A)$ with $\mathrm{E}_{A}=\big\{\{i^\ast_1,j^\ast_1\},\dots,\{i^\ast_l,j^\ast_l\}\big\}$. The following lemma holds \cite{wang2020}.

\begin{lemma}\label{so.2}
Consider a subset $\mathcal {S}\subseteq\mathpzc{B}$ with the associated graph $\mathrm{G}_{\mathcal {S}}:=\ell(\mathcal {S})$ and a matrix $A\in\so(n)$ with the associated graph $\mathrm{G}_A$. The following statements hold.

$(i)$ The Lie algebra generated by $\{A\}\mcup\mathcal {S}$ is equal to $\so(n)$ if each connected component of $\mathrm{G}_{\mathcal {S}}$ contains at least three nodes and the union graph $\mathrm{G}_A \mcup \mathrm{G}_{\mathcal {S}}$ is connected.

$(ii)$ If the Lie algebra generated by $\{A\}\mcup\mathcal {S}$ is equal to $\so(n)$, then the union graph $\mathrm{G}_A \mcup \mathrm{G}_{\mathcal {S}}$ is connected.
\end{lemma}

We are now in a position to show the detailed proof of Theorem \ref{thm1}. Since the union graph $\mathrm{G}_{\rm drift}^{{}^\alpha} \mcup \mathrm{G}_{\rm contr}^{{}^\beta}$ is connected, there must exist a matrix $A\in\Sigma_{\rm r}(\so(n)_{\alpha})$ such that $\mathrm{G}_A \mcup \mathrm{G}_{\rm contr}^{{}^\beta}$ is connected. Moreover, each connected component of $\mathrm{G}_{\rm contr}^{{}^\beta}$ contains at least three nodes. Lemma \ref{so.2}. (i) now leads to $\big\{\{A\}\mcup \so(n)_\beta\big\}_{\rm LA}=\so(n)$. Therefore, there exist $A\in\Sigma_{\rm r}(\so(n)_{\alpha})$ and $\so(n)_\beta\subset \Sigma_{\rm f}(\su(n)_{\beta})$, such that system \eqref{bilinear} is controllable on the Lie group ${\rm SO}(n)$. This is to say, the system \eqref{bilinear} is structurally controllable on the Lie group ${\rm SO}(n)$ with respect to the pair of zero patterns $\Sigma_{\rm zero}=\big(\Sigma_{\rm r}(\so(n)_{\alpha}),\Sigma_{\rm f}(\so(n)_{\beta})\big)$. This is the desired conclusion.

\subsection*{B. Proof of Theorem \ref{thm2.1}}
Before presenting the detailed proof for Theorem \ref{thm2.1}, we recall some auxiliary lemmas and definitions which will be useful for the analysis.
\begin{lemma}\label{lem1}
The Lie bracket of $E_{ij}$ and $E_{kl}$ in $\mathpzc{E}$ satisfies the relation $[E_{ij},E_{kl}] = \delta_{jk}E_{il}-\delta_{li}E_{kj}$, where $\delta$ is the Kronecker delta function, i.e., $\delta_{mn}=1$ if $m=n$ and $\delta_{mn}=0$ otherwise.
\end{lemma}

\begin{definition}\label{def7}
Let $\mathcal{G}$ be a simple digraph. The simple digraph transitive closure mapping, $\mathcal{M}(\cdot)$, is defined as
$$
\mathcal{M}(\mathcal{G}) = \Big(\mathrm{V}, \mathcal{E}\mcup\big\{(i,k)\,:\,\exists j \ \text{s.t. }\ (i,j)\in \mathcal{E},~(j,k)\in \mathcal{E},\ i\neq k \big\}\Big).
$$
We then recursively define $\mathcal{M}^{k}(\mathcal{G}):=\mathcal{M}(\mathcal{M}^{k-1}(\mathcal{G}))$ with $\mathcal{M}^{0}(\mathcal{G}):=\mathcal{G}$.
\end{definition}

\begin{lemma}\label{map.2}
Let $\mathcal{G}$ be a simple digraph. There exists an integer $z$ such that the digraph $\mathcal{M}^{z}(\mathcal{G})$ is simple complete if and only if $\mathcal{G}$ is strongly connected.
\end{lemma}

Consider $\mathfrak{D}$ as the set of all digraphs over the node set $\mathrm{V}$, and $\mathfrak{E}$ as the set of all subsets of $\mathpzc{E}$. By identifying each arc $(i,j)$ with the matrix $E_{ij}\in\mathpzc{E}$, we establish a natural $1-1$ correspondence between each element in $\mathfrak{D}$ and each element in $\mathfrak{E}$. We denote such a mapping $\tau$ which maps from $\mathfrak{E}$ to $\mathfrak{D}$. Denote $\mathpzc{E}_{1}=\{E_{ij}:1\le i\neq j\le n\}$, $\mathpzc{E}_{2}=\{E_{ii}:1\le i\le n\}$. Then with $\mathpzc{E}=\mathpzc{E}_{1}\mcup\mathpzc{E}_{2}$, for any $\mathcal {S}\subseteq \mathpzc{E}_{1}$, the digraph $\tau(\mathcal {S})$ is simple. The following lemmas, Lemma \ref{gl.0}, \ref{gl.1}, \ref{lem2} appeared in \cite{arxiv}. We therefore collect them below and the details of their proofs are omitted.

\begin{lemma}\label{gl.0}
Given a subset $\mathcal {S}\subseteq \mathpzc{E}_{1}$, let $\mathcal {S}^0\subseteq \mathcal {S}^1 \subseteq \mathcal {S}^2\cdots$ be an ascending chain of subsets of $\mathpzc{E}_{1}$ such that $\mathcal {S}^0=\mathcal {S}, \mathcal {S}^1 = [\mathcal {S}^0,\mathcal {S}^0]\mcup \mathcal {S}^0, \dots, \mathcal {S}^{k+1} = [\mathcal {S}^k, \mathcal {S}^k]\mcup \mathcal {S}^k,\dots.$
where $[\mathcal {S}^k, \mathcal {S}^k] = \big\{[E_{ij}, E_{jk}]=E_{ik}: E_{ij},E_{jk} \in \mathcal {S}^k ~{\rm and}~ i\neq k\big\}$. Then $\mathcal{M}^{k}(\tau(\mathcal {S}))=\tau(\mathcal {S}^k)$ holds for all $k= 0, 1,\dots$.
\end{lemma}

We denote by $\mathfrak{sl}(n)$ the vector space of $n\times n$ real traceless matrices.

\begin{lemma}\label{gl.1}
Let $\mathcal {S}$ be a subset of $\mathpzc{E}$ and $\tau(\mathcal {S})$ be the digraph associated
with $\mathcal {S}$.

$(i)$ The Lie algebra generated by $\mathcal {S}$ contains $\mathfrak{sl}(n)$ if and only if $\tau(\mathcal {S})$ is strongly connected.

$(ii)$ The Lie algebra generated by $\mathcal {S}$ is equal to $\gl(n)$ if and only if $\tau(\mathcal {S})$ is a strongly connected digraph with at least one self-loop.
\end{lemma}

\begin{definition}\label{def8}
Let $\mathcal{G}=(\mathrm{V},\mathcal{E})$ be a simple digraph. Given an ordered pair of nodes
$\langle i,j\rangle$ with $i, j\in\mathrm{V}$, the digraph
$\mathcal{H}_{ij}(\mathcal{G}) := (\mathrm{V},\mathcal{E}_{ij})$ is called the circumjacent closure at
$\langle i,j\rangle$ of $\mathcal{G}$ with
$
\mathcal{E}_{ij}=\mathcal{E}_{ij}^{1}\mcup\mathcal{E}_{ij}^{2}$, where
\begin{align*}
\mathcal{E}_{ij}^{1}= \big\{(i,k)\,:\,\exists k \ \text{s.t. }\ (j,k)\in \mathcal{E},~k\neq i\big\},
~~\mathcal{E}_{ij}^{2}= \big\{(k,j)\,:\,\exists k \ \text{s.t. }\ (k,i)\in \mathcal{E},~k\neq j\big\}.
\end{align*}
\end{definition}

\begin{lemma}\label{lem2}
 Let $\mathcal{G}=(\mathrm{V},\mathcal{E})$ be a simple directed graph. Suppose for $i,j\in \mathrm{V}$ we have $(i,j), (j,i)\notin \mathcal{E}$ and $\deg^{+}({i})=k$,\ $\deg^{-}({j})=l$. Then
$\mathcal{H}_{{i}{j}}(\mathcal{G})= (\mathrm{V},\mathcal{E}_{{i}{j}})$ satisfies

$(i)$ $|\mathcal{E}_{{i}{j}}|=k+l$;

$(ii)$ $\deg^{-}({i})=l$,\ $\deg^{+}({j})=k$;

$(iii)$ all nodes have zero degree except for ${i}$, ${j}$, ${i}$'s in-neighbors and ${j}$'s out-neighbors.
\end{lemma}
For any $A\in\gl(n)$, we can represent it uniquely in the form 
$$
A=\sum_{k=1}^{l} a_{k} E_{{i}^\ast_k{j}^\ast_k},
$$
where $a_k\neq 0\in\mathbb{R}$, $i^\ast_k,j^\ast_k\in\mathrm{V}$, and $E_{{i}^\ast_k{j}^\ast_k}\in\mathpzc{E}$. Define the map $\varphi$ that takes a matrix $A\in\gl(n)$ to a digraph $\mathcal{G}_{A}:=(\mathrm{V},\mathcal{E}_{A})$ with $\mathcal{E}_{A}=\big\{(i^\ast_1,j^\ast_1),\dots,(i^\ast_l,j^\ast_l)\big\}$. We will have established Theorem \ref{thm2.1} if we prove the lemma below.

\begin{lemma}\label{gl.2}
Consider a subset $\mathcal {S}\subseteq\mathpzc{E}$ with the associated digraph $\tau(\mathcal {S})$ and a matrix $A\in\gl(n)$ with the associated digraph $\mathcal{G}_A=(\mathrm{V},\mathcal{E}_A)$. Let $\mathcal{G}_{\mathcal {S}}=(\mathrm{V},\mathcal{E}_{\mathcal {S}})$ be $\tau(\mathcal {S})$. Then the Lie algebra generated by $\{A\}\mcup\mathcal {S}$ is equal to $\gl(n)$ if the following conditions hold:

$(i)$ Each weakly connected component of $\mathcal{G}_{\mathcal {S}}$ is strongly connected with at least two nodes;

$(ii)$ The digraph $\mathcal{G}_{\mathcal {S}}$ has at least one self-loop;

$(iii)$ The union graph $\mathcal{G}_A \mcup\mathcal{G}_{\mathcal {S}}$ is strongly connected.
\end{lemma}

\begin{proof}
Let $\widetilde{\mathcal{G}}_{A}=(\mathrm{V},\widetilde{\mathcal{E}}_A)$ and $\widetilde{\mathcal{G}}_{\mathcal {S}}=(\mathrm{V},\widetilde{\mathcal{E}}_{\mathcal {S}})$ be the simple digraphs corresponding to $\mathcal{G}_A$ and $\mathcal{G}_{\mathcal {S}}$ by ignoring the self-loops, respectively.
Because the union graph $\mathcal{G}_A \mcup\mathcal{G}_{\mathcal {S}}$ is strongly connected, the union graph $\widetilde{\mathcal{G}}_A \mcup\widetilde{\mathcal{G}}_{\mathcal {S}}$ is strongly connected too.
If $\mathcal{G}_{\mathcal {S}}$ is strongly connected, then by Lemma \ref{gl.1} the Lie algebra generated by $\{A\}\mcup\mathcal {S}$ is equal to $\gl(n)$. Now, assume that $\mathcal{G}_{\mathcal {S}}$ is the union of $m$ weakly connected components with $ m\geq2$.

Let $\widetilde{\mathcal{G}}_{\mathcal {S}}^{i}=(\mathrm{V}_{i},\widetilde{\mathcal{E}}_{\mathcal {S}}^{i})$ denote the $i$-th weakly connected component of  $\widetilde{\mathcal{G}}_{\mathcal {S}}$ for $i=1,\dots,m$. It is easily seen that each $\widetilde{\mathcal{G}}_{\mathcal {S}}^{i}$ is strongly connected. According to Lemma~\ref{map.2}, there exists an integer $z_{i}$ such that $\mathcal{M}^{z_{i}}(\widetilde{\mathcal{G}}_{\mathcal {S}}^{i})$ is a simple complete digraph for each $i$. Therefore $\mathcal{M}^{z^{*}}(\widetilde{\mathcal{G}}_{\mathcal {S}})=\mcup_{i=1}^{m}\mathcal{M}^{z_{i}}(\widetilde{\mathcal{G}}_{\mathcal {S}}^{i})$, where  $z^{*}=\max_i\{z_{i}\}$. For simplicity, we denote $\mathcal{G}^{*}=(\mathrm{V},\mathcal{E}^{*})=\mathcal{M}^{z^{*}}(\widetilde{\mathcal{G}}_{\mathcal {S}})$ with $\mathcal{E}^{*}=\mcup_{i=1}^{m}\{(u,v): u, v\in\mathrm{V}_{i}, u\neq v\}$. Lemma \ref{gl.0} shows that the elements in $\{E_{ij}:(i,j)\in\mathcal{E}^{*}\}$ can be generated by iterated Lie brackets of elements in $\mathcal {S}$.

Define the digraph $\mathcal{G}_{\rm valid}$ by $\mathcal{G}_{\rm valid}:= (\mathrm{V},\mathcal{E}_{\rm valid})$, where $\mathcal{E}_{\rm valid}:=\widetilde{\mathcal{E}}_{A}\setminus \mathcal{E}^{*}$.
As the union graph $\widetilde{\mathcal{G}}_A\mcup\widetilde{\mathcal{G}}_{\mathcal {S}}$ is strongly connected while $\widetilde{\mathcal{G}}_{\mathcal {S}}$ is not, we always have $|\mathcal{E}_{\rm valid}|\geq 2$. Because $\widetilde{\mathcal{G}}_A \mcup\widetilde{\mathcal{G}}_{\mathcal {S}}$ is strongly connected, $\mathcal{G}_{\rm valid} \mcup \widetilde{\mathcal{G}}_{\mathcal {S}}$ is also strongly connected by the definition of $\mathcal{G}_{\rm valid}$. In addition, we can deduce that $\mathcal{G}_{\rm valid}$ satisfies: i) all arcs are between different $\mathrm{V}_{i}$ and no arcs within each $\mathrm{V}_{i}$; ii) each $\mathrm{V}_{i}$ has at least one node with out-degree greater than zero; iii) each $\mathrm{V}_{i}$ has at least one node with in-degree greater than zero.

Let $\widetilde{A}=\sum_{({i}^\ast_k,{j}^\ast_k)\in\mathcal{E}_{\rm valid}}a_{k}E_{{i}^\ast_k{j}^\ast_k}.$ It is clear that $\varphi(\widetilde{A})=\mathcal{G}_{\rm valid}$. The Lie bracket of $\widetilde{A}$ and $E_{ij}\in\{E_{ij}:(i,j)\in\mathcal{E}^{*}\}$ satisfies $[\widetilde{A}, E_{ij}]\in\big\{\{A\}\mcup\mathcal {S}\big\}_{\rm LA}$. In addition, the relationship between the Lie bracket $[\widetilde{A},E_{ij}]$ and the circumjacent closure at $\langle i,j\rangle$ of $\mathcal{G}_{\rm valid}$ satisfies
\begin{equation}\label{eq1}
\varphi\big([\widetilde{A},E_{ij}] \big)=\mathcal{H}_{ij}(\mathcal{G}_{\rm valid})
\end{equation}
for $ E_{ij}\in \{E_{ij}:(i,j)\in\mathcal{E}^{*}\}$. To prove the statement, we need to consider the Lie algebra generated by $\{A\}\mcup\mathcal {S}$. Since $\mathcal{G}_{\mathcal {S}}$ has at least one self-loop, without loss of generality we assume that the node $v_{11}\in\mathrm{V}_{1}$ has a self-loop, i.e., $E_{v_{11}v_{11}}\in\mathcal {S}$. The remainder of the proof contains two steps.

\noindent{\it Step 1.} We first prove that if $\mathcal{G}_{\rm valid}$ has arcs from the nodes in $\mathrm{V}_{1}$ to the nodes in $\mathrm{V}_{k}, k\in \{2,\dots,m\}$,
then  all elements in the set $\{E_{ij}:i\in\mathrm{V}_{1},~ j\in\mathrm{V}_{k}\}$ can be obtained by iterated Lie brackets of elements in $\{\widetilde{A}\}\mcup\mathcal {S}$. Note that for any node $v\in\mathrm{V}_{1}$ in  $\mathcal{G}_{\rm valid}$ with $\deg^-(v)>0$, Lemma \ref{gl.1} yields $E_{vv}\in\mathcal \{{S}\}_{\rm LA}$ since $\mathcal{G}_{\mathcal {S}}^{1}$ is a strongly connected digraph with self-loops. We only need to consider the case where the out-degree of node $v_{11}$ is greater than zero. The analysis for the other nodes in $\mathrm{V}_{1}$ whose out-degree is greater than zero can be similarly established.

Let $\deg^-(v_{11})=k>0$, and $v_{i_{1}j_{1}},\dots,v_{i_{k}j_{k}}$ be the out-neighbors of $v_{11}$. Apparently, these out-neighbors are in $\mathrm{V}_{2}\mcup\dots\mcup\mathrm{V}_{m}$. To be specific, let $v_{i_{1}j_{1}},\dots,v_{i_{r}j_{r}}\in\mathrm{V}_{2}$ with $1\leq r\leq k$. Applying Lemma \ref{lem1}, we deduce that the digraph $\varphi\big([\widetilde{A},E_{v_{11}v_{11}}]\big)$ satisfies: i) $\deg^-(v_{11})=k>0$; ii) other nodes in $\mathrm{V}_{1}$ has zero degree. Fix $v_{12}\in{\rm V}_1$ and $v_{12}\neq v_{11}$. Using Lemma \ref{lem2}, we have
\begin{equation*}
  \mathcal{H}_{v_{12}v_{11}}\Big(\varphi\big([\widetilde{A},E_{v_{11}v_{11}}]\big)\Big)=\Big(\mathrm{V},\big\{(v_{12},v_{i_{1}j_{1}}),(v_{12},v_{i_{2}j_{2}}),\dots,(v_{12},v_{i_{k}j_{k}})\big\}\Big).
  \end{equation*}
Let $v_{i_{*}j_{*}}\neq v_{i_{1}j_{1}} \in \mathrm{V}_{2}$. By selecting the node pair $\langle v_{i_{1}j_{1}},v_{i_{*}j_{*}}\rangle$ we can obtain
\begin{equation}\label{eq2}
 \mathcal{H}_{v_{i_{1}j_{1}}v_{i_{*}j_{*}}}\Big(\mathcal{H}_{v_{12}v_{11}}\big(\varphi([\widetilde{A},E_{v_{11}v_{11}}])\big)\Big)= \Big(\mathrm{V},\big\{(v_{12},v_{i_{*}j_{*}})\big\}\Big).
\end{equation}
Since $(v_{12},v_{11})$ and $(v_{i_{1}j_{1}},v_{i_{*}j_{*}})$ are in $\mathcal{E}^{*}$, i.e., $E_{v_{12}v_{11}},E_{v_{i_{1}j_{1}}v_{i_{*}j_{*}}}\in\{\mathcal {S}\}_{\rm LA}$, from \eqref{eq1} and \eqref{eq2} we conclude that
$$
\varphi\Big([[[\widetilde{A},E_{v_{11}v_{11}}],E_{v_{12}v_{11}}],E_{v_{i_{1}j_{1}}v_{i_{*}j_{*}}}]\Big)
=\Big(\mathrm{V},\big\{(v_{12},v_{i_{*}j_{*}})\big\}\Big). $$
This implies that
$$[[[\widetilde{A},E_{v_{11}v_{11}}],E_{v_{12}v_{11}}],E_{v_{i_{1}j_{1}}v_{i_{*}j_{*}}}]= a^{*}E_{v_{12}v_{i_{*}j_{*}}},$$
where $a^{*}$ is the coefficient generated during the operation of the Lie brackets.

Therefore,
$E_{v_{12}v_{i_{*}j_{*}}}\in\mathpzc{E}_{1}$ can be obtained by iterated Lie brackets of elements in $\{\widetilde{A}\}\mcup\mathcal {S}$.
This, together with the strong connectivity of $\mathcal{G}_{\mathcal {S}}^{1}$ and $\mathcal{G}_{\mathcal {S}}^{2}$, implies that all elements in the set  $\{E_{ij}:i\in\mathrm{V}_{1}, ~j\in\mathrm{V}_{2}\}$ can be obtained by iterated Lie brackets of elements in $\{\widetilde{A}\}\mcup\mathcal {S}$.

\noindent{\it Step 2.} Analysis similar to that in the proof of Theorem 3 in \cite{arxiv} shows that all elements in set  $\mathcal {S}^*:= \{E_{ij}:i\in\mathrm{V}_{1},~j\in\mathrm{V}_{2}\mcup\dots\mcup\mathrm{V}_{m}\}\mcup\{E_{ij}:i\in\mathrm{V}_{2}\mcup\dots\mcup\mathrm{V}_{m},~ j\in\mathrm{V}_{1}\}$ can be obtained by iterated Lie brackets of elements in $\{\widetilde{A}\}\mcup\mathcal {S}$. Since $\mathcal {S}^*\mcup \mathcal {S}\subseteq\mathpzc{E}$ and $\tau(\mathcal {S}^*\mcup \mathcal {S})$ is a connected digraph with self-loops, we conclude from Lemma \ref{gl.1} that $\{\mathcal {S}^*\mcup \mathcal {S}\}_{\rm LA}=\gl(n)$, and consequently, we have $\big\{\{A\}\mcup\mathcal {S}\big\}_{\rm LA}=\gl(n)$. This completes the proof.
\end{proof}

\subsection*{C. Proof of Theorem \ref{thm2.2}}
The statement in Theorem \ref{thm2.2} will be proved once we prove the lemma below.
\begin{lemma}\label{gl.3}
Consider two subset $\mathcal {A}, \mathcal {S}\subseteq \mathpzc{E}$ with the associated graphs $\mathcal{G}_{\mathcal {A}}=\tau(\mathcal {A})$ and $\mathcal{G}_{\mathcal {S}}=\tau(\mathcal {S})$, respectively. There exists a matrix $A\in\Sigma_{\rm r}(\mathcal {A})$ such that the Lie algebra generated by $\{A\}\mcup\mathcal{S}$ is equal to $\gl(n)$ if the following conditions hold:

$(i)$ Each weakly connected component of $\mathcal{G}_{\mathcal {S}}$ is strongly connected with at least two nodes;

$(ii)$ The union graph $\mathcal{G}_{\mathcal {A}}\mcup\mathcal{G}_{\mathcal {S}}$ is a strongly connected digraph with self-loops.
\end{lemma}
\begin{proof}
Note that $\mathcal{G}_{\mathcal {A}}\mcup\mathcal{G}_{\mathcal {S}}$ has self-loops. That is, $\mathcal{G}_{\mathcal {A}}$ or $\mathcal{G}_{\mathcal {S}}$ has self-loops. Using Lemma \ref{gl.2}, for any $A\in\Sigma_{\rm r}(\mathcal {A})$, the generated Lie algebra by $\{A\}\mcup\mathcal{S}$ is equal to $\gl(n)$ if $\mathcal{G}_{\mathcal {S}}$ has self-loops. Thus, in the rest of the proof, we consider the case where only $\mathcal{G}_{\mathcal {A}}$ has self-loops. Set  $\mathcal {A}=\{E_{i_1i_1},E_{i_2i_2},\dots,E_{i_li_l},E_{i_{l+1}j_{l+1}},\dots,E_{i_{l+r}j_{l+r}}\}$ with $l\geq 1$. For any $A\in\Sigma_{\rm r}(\mathcal {A})$,
$$
A=\sum_{k=1}^l a_{k} E_{i_ki_k}+\sum_{k=l+1}^{l+r} a_{k} E_{i_kj_k},
$$
where $a_{k}\neq 0\in\mathbb{R}$. It follows that $\mathcal{G}_{A}=\mathcal{G}_{\mathcal {A}}$.

If $\mathcal{G}_{\mathcal {S}}$ is strongly connected, then by Lemma \ref{gl.1} the Lie algebra generated by $\mathcal {S}$ contains $\mathfrak{sl}(n)$. Hence, for any $A\in\Sigma_{\rm r}(\mathcal {A})$, there holds $\sum_{k=1}^{l} a_{k} E_{i_ki_k}\in \big\{\{A\}\mcup\mathcal{S}\big\}_{\rm LA}.$ Because
$E_{ii}-E_{jj}\in\{\mathcal {S}\}_{\rm LA}$, for $1\leq i\neq j \leq n$, we can obtain
\begin{equation}\label{eq3}
    \begin{split}
    \sum_{k=1}^{l}a_{k}E_{i_{k}i_{k}} & +a_{l}(E_{i_{l-1}i_{l-1}}-E_{i_{l}i_{l}}) \\
    & +(a_{l}+a_{l-1})(E_{i_{l-2}i_{l-2}}-E_{i_{l-1}i_{l-1}})\\
    & +\cdots+\Big(\sum_{k=2}^{l}a_{k}\Big)(E_{i_{1}i_{1}}-E_{i_{2}i_{2}})\\
    & =\Big(\sum_{k=1}^{l}a_{k}\Big)E_{i_{1}i_{1}}\in \big\{\{A\}\mcup\mathcal{S}\big\}_{\rm LA}.
    \end{split}
\end{equation}
Therefore, when ${\rm tr}A\neq 0$, i.e., $\sum_{k=1}^{l}a_{k}\neq0$, we have $E_{i_{1}i_{1}}\in\big\{\{A\}\mcup\mathcal{S}\big\}_{\rm LA}.$ Lemma \ref{gl.1} now yields $\big\{\{A\}\mcup\mathcal{S}\big\}_{\rm LA}=\gl(n)$. Consequently, for all $A\in\Sigma_{\rm r}(\mathcal {A})$ with ${\rm tr}A\neq 0$, the Lie algebra generated by $\{A\}\mcup\mathcal{S}$ is equal to $\gl(n)$.

 Now, let $\mathcal{G}_{\mathcal {S}}$ be the union of $m$ weakly connected components with $ m\geq2$. Consider $A\in\Sigma_{\rm r}(\mathcal {A})$ with ${\rm tr}A\neq 0$. We continue to use the definitions of $\widetilde{\mathcal{G}}_{A}$, $\widetilde{\mathcal{G}}_{\mathcal {S}}$, $\mathcal{G}^{*}$, $\mathcal{G}_{\rm valid}$ and $\widetilde{A}$ in the proof of Lemma \ref{gl.2}. By the definition of $\widetilde{A}$, one has
 $$\widetilde{A}+\sum_{k=1}^{l} a_{k} E_{i_ki_k}\in \big\{\{A\}\mcup\mathcal{S}\big\}_{\rm LA}.$$
 Recall that all the elements in $\{E_{ij}:(i,j)\in\mathcal{E}^{*}\}$ can be generated by iterated Lie brackets of elements in $\mathcal {S}$. Using Lemma \ref{lem1}, we have
\begin{equation}\label{eq11}
 \{\mathcal{S}\}_{\rm LA}\supset\{E_{ij}:~(i,j)\in\mathcal{E}^{*}\}\mcup\{E_{ii}-E_{jj}:~(i,j)\in\mathcal{E}^{*}\}.
\end{equation}
Without loss of generality, we assume $\{i_1,\dots,i_{l_1}\}\in\mathrm{V}_1$, $\{i_{l_1+1},\dots,i_{l_2}\}\in\mathrm{V}_2,\dots, \{i_{l_{r}+1},\dots,i_{l_{r+1}}=i_{l}\}\in\mathrm{V}_{r+1}$, with $1\leq r+1\leq m$. Repeating the process in \eqref{eq3} for $\widetilde{A}+\sum_{k=1}^{l} a_{k} E_{i_ki_k}$ gives
$$\widetilde{A}+\sum_{k=1}^{r+1} a_{k}^* E_{i_{l_k}i_{l_k}}\in \big\{\{A\}\mcup\mathcal{S}\big\}_{\rm LA},$$
where $a_{k}^*\neq 0$ if and only if $\sum_{j=l_{k-1}+1}^{l_k} a_j\neq0$. There must exist a $1\leq k\leq r+1$ such that $a_{k}^*\neq 0$ since ${\rm tr}A\neq 0$. From now on, we will write it simply as $\widetilde{A}+\sum_{k=1}^{r+1} a_{k} E_{i_{k}i_{k}}$ when no confusion can arise. It is worth pointing out that node $i_k$ is in $\mathrm{V}_k$ for $k=1,\dots,r+1.$ The remainder of the proof will be divided into three steps.

\noindent{\it Step 1.} We first prove that the following relation holds:
\begin{equation}\label{eq4}
\sum_{k=1}^{r+1} a_{k} E_{i_ki_k}\in \big\{\{A\}\mcup\mathcal{S}\big\}_{\rm LA}.
\end{equation}
To this end, consider $E_{i^*j^*}\in \{E_{ij}:(i,j)\in\mathcal{E}^{*}\}$. The matrix $\widetilde{A}$ can be decomposed into $\widetilde{A}=\widetilde{A}_1+\widetilde{A}_2+\widetilde{A}_3$, where $\widetilde{A}_1$ is the part of the linear combination of elements of the form $E_{ri^*}$ in $\widetilde{A}$, $\widetilde{A}_2$ is the part of the linear combination of elements of the form $E_{j^*t}$ in $\widetilde{A}$, and the remaining part is $\widetilde{A}_3$. The proof of \eqref{eq4} is based on the following computation.
\begin{align*}
  [[\widetilde{A}+\sum_{k=1}^{r+1} a_{k} E_{i_{k}i_{k}},E_{i^*j^*}],E_{j^*i^*}] =&  [[\widetilde{A}_1+\widetilde{A}_2+\widetilde{A}_3+\sum_{k=1}^{r+1} a_{k} E_{i_{k}i_{k}},E_{i^*j^*}],E_{j^*i^*}] \\
  = & \widetilde{A}_1+\widetilde{A}_2+\delta_{i_ki^*}a_k(E_{i^*i^*}-E_{j^*j^*})+\delta_{i_kj^*}a_k(E_{j^*j^*}-E_{i^*i^*}).
\end{align*}
Set $A'=\delta_{i_ki^*}a_k(E_{i^*i^*}-E_{j^*j^*})+\delta_{i_kj^*}a_k(E_{j^*j^*}-E_{i^*i^*})$. Of course, $A'\in \big\{\{A\}\mcup\mathcal{S}\big\}_{\rm LA}$. Then there holds
\begin{align*}
 & \widetilde{A}+\sum_{k=1}^{r+1} a_{k} E_{i_{k}i_{k}}-\big((\widetilde{A}_1+\widetilde{A}_2+A')-A'\big)\\ =& \widetilde{A}_1+\widetilde{A}_2+\widetilde{A}_3+\sum_{k=1}^{r+1} a_{k} E_{i_{k}i_{k}}- (\widetilde{A}_1+\widetilde{A}_2)\\
  = & \widetilde{A}_3+\sum_{k=1}^{r+1} a_{k} E_{i_{k}i_{k}}\in \big\{\{A\}\mcup\mathcal{S}\big\}_{\rm LA}.
 \end{align*}
The term $\widetilde{A}_3+\sum_{k=1}^{r+1} a_{k} E_{i_{k}i_{k}}$ can be handled in much the same way, the only difference being in the choice of $E_{i^*j^*}$. Repeating this process, we can finally conclude that $\sum_{k=1}^{r+1} a_{k} E_{i_{k}i_{k}}\in \big\{\{A\}\mcup\mathcal{S}\big\}_{\rm LA}$.

\noindent{\it Step 2.} We next prove that for any $v_k\in\mathrm{V}_k, k=1,\dots,r+1$, there holds
\begin{equation}\label{eq5}
\sum_{k=1}^{r+1} a_{k} E_{v_{k}v_{k}}\in \big\{\{A\}\mcup\mathcal{S}\big\}_{\rm LA}.
\end{equation}
Recall that $i_k\in\mathrm{V}_k$, for $k=1,\dots,r+1$. Based on \eqref{eq11} and \eqref{eq4}, it can be concluded that 
\begin{equation*}
\sum_{k=1}^{r+1} a_{k} E_{i_{k}i_{k}}+a_1(E_{v_1v_1}-E_{i_1i_1})=a_1E_{v_1v_1}+\sum_{k=2}^{r+1} a_{k} E_{i_{k}i_{k}}\in\big\{\{A\}\mcup\mathcal{S}\big\}_{\rm LA},
\end{equation*}
for any $v_1\in\mathrm{V}_{1}$ and $v_1\neq i_1$. It is sufficient to show that the statement in \eqref{eq5} holds.

\noindent{\it Step 3.} In this step, we prove that there must exist a matrix $A^*\in\Sigma_{\rm r}(\mathcal {A})$ such that $\big\{\{A^*\}\mcup\mathcal{S}\big\}_{\rm LA}=\gl(n)$. As the {\it step 1} in the proof of Lemma \ref{gl.2}, we consider the digraph $\mathcal{G}_{\rm valid}$. Let $v_{11}$ be the node in $\mathrm{V}_{1}$ with $\deg^-(v_{11})=l>0$, and $v_{i_{1}j_{1}},\dots,v_{i_{l}j_{l}}$ be the out-neighbors of $v_{11}$. Let $v_{k1}$ and $v_{k2}$ denote two different nodes in $\mathrm{V}_{k}$, for $k=1,\dots,r+1$.
From \eqref{eq5} we obtain
$$a_1E_{v_{11}v_{11}}+\sum_{k=2}^{r+1} a_{k} E_{v_{k1}v_{k1}}\in \big\{\{A\}\mcup\mathcal{S}\big\}_{\rm LA},$$
and
$$a_1E_{v_{11}v_{11}}+\sum_{k=2}^{r+1} a_{k} E_{v_{k2}v_{k2}}\in \big\{\{A\}\mcup\mathcal{S}\big\}_{\rm LA}.$$
Write $\mathcal{G}=\varphi\big([a_1E_{v_{11}v_{11}}+\sum_{k=2}^{r+1} a_{k} E_{v_{k2}v_{k2}},[a_1E_{v_{11}v_{11}}+\sum_{k=2}^{r+1} a_{k} E_{v_{k1}v_{k1}}, \widetilde{A}]]\big)$. The digraph $\mathcal{G}$ satisfies: i) $\deg(v_{12})=0$; ii) all edges are adjacent to node $v_{11}$; iii) the set of out-neighbors of node $v_{11}$ is a subset of $\{v_{i_{1}j_{1}},\dots,v_{i_{l}j_{l}}\}$.
Let $\mathcal {O}$ be the set of out-neighbors of node $v_{11}$ in $\mathcal{G}$. Whether the set $\mathcal {O}$ and the set $\{v_{i_{1}j_{1}},\dots,v_{i_{l}j_{l}}\}$ are equal is related to the combination coefficient of $A$. It follows that there must exist a matrix $A^*\in\Sigma_{\rm r}(\mathcal {A})$ such that $\mathcal {O}=\{v_{i_{1}j_{1}},\dots,v_{i_{l}j_{l}}\}$. Proceeding as in the proof of Lemma \ref{gl.2}, we have $\big\{\{A^*\}\mcup\mathcal{S}\big\}_{\rm LA}=\gl(n)$, and the proof is complete.
\end{proof}

\subsection*{D. Proof of Proposition \ref{su.1}}

\begin{definition}\label{map.1}
$(i)$ Let ${\rm S}_1\subseteq\mathpzc{B}$, $\mathrm{S}_2\subseteq\mathpzc{C}$ and $\mathrm{S}_3\subseteq\mathpzc{D}$. The edge-colored multigraph $\mathscr{G}_{\mathrm{S}_1\mcup\mathrm{S}_2\mcup\mathrm{S}_3}$ is given by $\mathscr{G}_{\mathrm{S}_1\mcup\mathrm{S}_2\mcup\mathrm{S}_3}=(\mathrm{V},\mathscr{E}_{\mathrm{S}_1\mcup\mathrm{S}_2\mcup\mathrm{S}_3})$ with $\mathscr{E}_{\mathrm{S}_1\mcup\mathrm{S}_2\mcup\mathrm{S}_3}=\mathscr{E}_{\mathrm{S}_1}\mcup\mathscr{E}_{\mathrm{S}_2}\mcup\mathscr{E}_{\mathrm{S}_3}$, where $\mathscr{E}_{\mathrm{S}_1}=\big\{\{i,j;{\rm Blue}\}:B_{ij}\in\mathrm{S}_1\big\}$, $\mathscr{E}_{\mathrm{S}_2}=\big\{\{i,j;{\rm Red}\}:C_{ij}\in\mathrm{S}_2\big\}$, and $\mathscr{E}_{\mathrm{S}_3}=\big\{\{k,k;{\rm Green}\}: k=i,j, D_{ij}\in\mathrm{S}_3\big\}$.

$(ii)$ Let $\mathscr {G}$ be an edge-colored multigraph without self-loops, i.e., there are ${\rm S}_1\subseteq\mathpzc{B}$ and $\mathrm{S}_2\subseteq\mathpzc{C}$, such that $\mathscr {G}=\mathscr{G}_{\mathrm{S}_1\mcup\mathrm{S}_2}$. Then the edge-colored multigraph transitive closure mapping, $\mathcal{\mathcal {T}}(\cdot)$, is defined as
$$
\mathcal{T}\big(\mathscr{G}_{\mathrm{S}_1\mcup\mathrm{S}_2}\big) = \big(\mathrm{V}, \mathscr{E}_{\mathrm{S}_1\mcup\mathrm{S}_2}\mcup\mathscr{E}_1\mcup\mathscr{E}_2\mcup\mathscr{E}_3\big),
$$
where $\mathscr{E}_1=\big\{\{i,k;{\rm Blue}\}:\,\exists j \ \text{s.t. }\ \{i,j;{\rm Blue}\}\in \mathscr{E}_{\mathrm{S}_1}, \{j,k;{\rm Blue}\}\in \mathscr{E}_{\mathrm{S}_1} \big\}$, $\mathscr{E}_2=\big\{\{i,k;{\rm Blue}\}:\,\exists j \ \text{s.t. }\ \{i,j;{\rm Red}\}\in \mathscr{E}_{\mathrm{S}_2}, \{j,k;{\rm Red}\}\in \mathscr{E}_{\mathrm{S}_2} \big\}$ and $\mathscr{E}_3=\big\{\{i,k;{\rm Red}\}:\,\exists j \ \text{s.t. }\ \{i,j;{\rm Red}\}\in \mathscr{E}_{\mathrm{S}_2}, \{j,k;{\rm Blue}\}\in \mathscr{E}_{\mathrm{S}_1}, i\neq k \big\}$.
\end{definition}
We then recursively define $\mathcal{T}^{k}(\mathscr{G}):=\mathcal{T}(\mathcal{T}^{k-1}(\mathscr{G}))$ with $\mathcal{T}^{0}(\mathscr{G}):=\mathscr{G}$.
It is worth pointing out that there exists an integer $z$ such that $K_n$ is a spanning subgraph of $\mathcal{T}^{z}(\mathscr{G})$ if and  only if $\mathscr{G}$ is connected, where $K_n$ is the complete graph of $n$ vertices.

\begin{proposition}\label{su.1}
Consider a subset ${\rm S}\subseteq\mathpzc{B}\mcup\mathpzc{C}\mcup\mathpzc{D}$ with the associated edge-colored multigraph $\mathscr{G}_{\rm S}$. Then the Lie algebra generated by ${\rm S}$ is equal to $\su(n)$ if and only if $\mathscr{G}_{\rm S}$ satisfies one of the following conditions:

$(i)$ $\mathscr{G}_{\rm S}$ has edges of at least two colors, and all the Blue edges together with the vertex set $\mathrm{V}$ form a connected graph;

$(ii)$ $\mathscr{G}_{\rm S}$ is a connected graph with self-loops;

$(iii) $ $\mathscr{G}_{\rm S}$ is a connected graph with a cycle containing an odd number of Red edges.
\end{proposition}

Our subsequent study will be based on the following lemma:
\begin{lemma}\label{lem3}
The Lie brackets among the basis elements
of $\su(n)$ satisfy
\[[B_{ij},B_{kl}] = \delta_{jk}B_{il}+\delta_{il}B_{jk}+\delta_{jl}B_{ki}+\delta_{ik}B_{lj};\]
\[[C_{ij},C_{kl}] =\delta_{li}B_{kj}+\delta_{ki}B_{lj}+\delta_{lj}B_{ki}+\delta_{kj}B_{li};\]
\[[B_{ij},C_{kl}]=\delta_{jk}C_{il}+\delta_{jl}C_{ik}-\delta_{il}C_{kj}-\delta_{ik}C_{lj};\]
\[[B_{ij},D_{kl}]=(\delta_{jk}+\delta_{li}-\delta_{ki}-\delta_{jl})C_{ij};\]
\[[C_{ij},D_{kl}]=(\delta_{ki}+\delta_{jl}-\delta_{kj}-\delta_{il}) B_{ij};\]
and
\[[D_{ij},D_{kl}]= 0.\]
for all $1\leq i,j,k,l \leq n.$
\end{lemma}
It is easy to prove this lemma directly, so we omit the proof.

Notice that, by Lemma \ref{lem3}, $ [B_{ij} , B_{kl}] \neq0$ if and only if there
exists a bridging index
\begin{equation}\label{eq6}
 j = k,~ i = l,~j = l,{~\rm or}~ i = k.
\end{equation}
 Likewise, $[C_{ij},C_{kl}]\neq 0, [B_{ij},C_{kl}]\neq0, [B_{ij},D_{kl}]\neq0$ and $[C_{ij},D_{kl}]\neq 0$ if and only if \eqref{eq6} holds. In particular, $[B_{ij},C_{ij}]=2D_{ij}$.

\subsubsection*{D.1 Proof of Sufficiency for Proposition \ref{su.1}}
Denote ${\rm S}={\rm S}_1\mcup{\rm S}_2\mcup{\rm S}_3$ with ${\rm S}_1\subseteq\mathpzc{B}$, ${\rm S}_2\subseteq\mathpzc{C}$ and ${\rm S}_3\subseteq\mathpzc{D}$. In this section, we show how the Lie algebra generated by ${\rm S}$ is equal to $\su(n)$ under conditions $(i),(ii)$ and $(iii)$ respectively.

(i) Suppose $\mathscr{G}_{\rm S}$ has edges of at least two colors, and all the Blue edges together with the vertex set $\mathrm{V}$ form a connected graph. It is clear that the edge generated subgraph obtained from all Blue edges is the graph $\mathscr{G}_{{\rm S}_1}$. From Lemma \ref{so.1}, we conclude that all the elements in $\mathpzc{B}$ can be generated by iterated Lie brackets of elements in ${\rm S}_1$, i.e., $\{{\rm S}_1\}_{\rm LA}=\so(n)$. Also $\mathscr{G}_{\rm S}$ has edges of at least two colors, i.e., ${\rm S}_2\mcup{\rm S}_3\neq\emptyset$. Note that $[B_{ij},C_{ij}]=2D_{ij}$. We only need to show that if ${\rm S}_3\neq\emptyset$, then $\{{\rm S}_1\mcup{\rm S}_3\}_{\rm LA}=\su(n)$.

There is no loss of generality in assuming $D_{12}\in{\rm S}_3$. Notice that $\mathpzc{B}\subset\{{\rm S}_1\}_{\rm LA}$. By Lemma \ref{lem3}, we have
$$[B_{12},D_{12}]=-2C_{12},~~[B_{1j},D_{12}]=-C_{1j},~ j=3,\dots,n.$$
Then it is evident that
$$\mathpzc{C}_1:=\{C_{1j}:2\leq j\leq n\}\subset\{{\rm S}_1\mcup{\rm S}_3\}_{\rm LA}.$$
In addition, $[B_{2j},D_{12}]=C_{2j}$ for $j=3,\dots,n$. This clearly forces
$$\mathpzc{C}_2:=\{C_{2j}:3\leq j\leq n\}\subset\{{\rm S}_1\mcup{\rm S}_3\}_{\rm LA}.$$
Lemma \ref{lem3} now gives $D_{23}\in\{{\rm S}_1\mcup{\rm S}_3\}_{\rm LA}$ since $[B_{23},C_{23}]=2D_{23}$. In the same manner, we can see that
$$\mathpzc{C}_3:=\{C_{3j}:4\leq j \leq n\}\subset\{{\rm S}_1\mcup{\rm S}_3\}_{\rm LA}.$$
We continue in this fashion obtaining $\mathpzc{C}_i:=\{C_{ij}:i+1\leq j \leq n\}\subset\{{\rm S}_1\mcup{\rm S}_3\}_{\rm LA}$ for $i=4,\dots,n-1$. Therefore $\mathpzc{C}:=\mcup_{i=1}^{n-1}\mathpzc{C}_i\subset\{{\rm S}_1\mcup{\rm S}_3\}_{\rm LA}.$ Immediately there holds $\mathpzc{D}\subset\{{\rm S}_1\mcup{\rm S}_3\}_{\rm LA}$, and consequently, we have $\{{\rm S}_1\mcup{\rm S}_3\}_{\rm LA}=\su(n)$.

(ii) Let $\mathscr{G}_{\rm S}$ be a connected graph with self-loops. Evidently, the subgraph $\mathscr{G}_{{\rm S}_1\mcup{\rm S}_2}:=\big(\mathrm{V},\mathscr{E}_{{\rm S}_1\mcup{\rm S}_2}\big)$ of $\mathscr{G}_{\rm S}$ is connected, and ${\rm S}_3\neq\emptyset$. We can certainly assume that $D_{12}\in {\rm S}_3$. Based on (i), if we prove that the Lie algebra $\{{\rm S}\}_{\rm LA}$ contains $\so(n)$ as a Lie subalgebra, the statement $\{{\rm S}\}_{\rm LA}=\su(n)$ follows.

It is worth noting from Lemma \ref{lem3} and Definition \ref{map.1} that taking first-order Lie brackets of the elements in
${\rm S}_1\mcup{\rm S}_2$ corresponds to adding edges that connect the endpoints of incident edges in $\mathscr{G}_{{\rm S}_1\mcup{\rm S}_2}$. As $\mathscr{G}_{{\rm S}_1\mcup{\rm S}_2}$ is connected, we conclude that there exists an integer $z$ such that $K_n$ is a spanning subgraph of $\mathcal{T}^z\big(\mathscr{G}_{\mathrm{S}_1\mcup\mathrm{S}_2}\big)$. In other words, for any $1\leq i<j\leq n$, $B_{ij}\in\{{\rm S}_1\mcup{\rm S}_2\}_{\rm LA}$ or $C_{ij}\in\{{\rm S}_1\mcup{\rm S}_2\}_{\rm LA}$.

Fix $i=1$. There are two cases: (a) if for all $2\leq j\leq n$, $B_{1j}\in\{{\rm S}_1\mcup{\rm S}_2\}_{\rm LA}$, then $\ell\big(\{B_{12},\dots,B_{1n}\}\big)$ is connected and thus $\{B_{12},\dots,B_{1n}\}_{\rm LA}=\so(n)$. It follows that $\so(n)\subset \{{\rm S}_1\mcup{\rm S}_2\}_{\rm LA}$; (b) if there is a $2\leq j^*\leq n$ such that $B_{1j^*}\notin\{{\rm S}_1\mcup{\rm S}_2\}_{\rm LA}$, then $C_{1j^*}$ must be in $\{{\rm S}_1\mcup{\rm S}_2\}_{\rm LA}$. Since $D_{12}\in{\rm S}_3$,
Lemma \ref{lem3} shows that $B_{1j^*}\in\{{\rm S}_1\mcup{\rm S}_2\mcup{\rm S}_3\}_{\rm LA}$, by taking Lie bracket of $C_{1j^*}$
and $D_{12}$. Hence, for any $2\leq j\leq n$, $B_{1j}\in\{{\rm S}_1\mcup{\rm S}_2\mcup{\rm S}_3\}_{\rm LA}$. This clearly forces $\so(n)\subset\{{\rm S}_1\mcup{\rm S}_2\mcup{\rm S}_3\}_{\rm LA}$.

In conclusion, the Lie algebra generated by ${\rm S}$ contains $\so(n)$ as a Lie subalgebra. Furthermore, with ${\rm S}_3\neq\emptyset$, we have $\{{\rm S}\}_{\rm LA}=\su(n)$.

(iii) Let $\mathscr{G}_{\rm S}$ be a connected graph with a cycle containing an odd number of Red edges. By (ii) it is immediate that if there are self-loops in $\mathscr{G}_{\rm S}$, then $\{{\rm S}\}_{\rm LA}=\su(n)$. Thus, we focus on proving $\{{\rm S}\}_{\rm LA}=\su(n)$ for $\mathscr{G}_{\rm S}$ without self-loops, i.e., ${\rm S}_3=\emptyset$.
Let ${\rm P}=\big(\{v_1,v_2\},\{v_2,v_3\},\dots,\{v_l,v_{l+1}=v_1\}\big)$ denote the cycle containing an odd number of Red edges. The set of matrices corresponding to these edges  in ${\rm P}$ will be denoted by ${\rm S_P}$. The proof is based on the following lemma.
\begin{lemma}\label{su.2}
The Lie algebra generated by ${\rm S_P}$ contains at least one element in $\mathpzc{D}$.
\end{lemma}
\begin{proof}
Recall that $[B_{ij},C_{ij}]=2D_{ij}, [B_{ij},B_{jk}]=B_{ik}, [C_{ij},C_{jk}]=-B_{ik}$ and $ [C_{ij},B_{jk}]=C_{ik}$. The proof is by induction on $l$, i.e., the cardinality of set ${\rm S_P}$. We divide our proof in three steps.

\noindent{\it Step 1.} Notice that the cardinality of set ${\rm S_P}\mcap\mathpzc{C}$ is an odd number. If $l=2$, then ${\rm S_P}=\{B_{v_1v_2},C_{v_1v_2}\}$. It follows easily that $D_{v_1v_2}\in\{{\rm S_P}\}_{\rm LA}$.

\noindent{\it Step 2.} Let $l=3$. There are two possibilities for elements in set ${\rm S_P}$:
i) ${\rm S_P}=\{B_{v_1v_2},B_{v_2v_3},C_{v_1v_3}\}$; ii) ${\rm S_P}=\{C_{v_1v_2},C_{v_2v_3},C_{v_1v_3}\}$. In either case, we have $B_{v_1v_3}\in\{{\rm S_P}\}_{\rm LA}$, which implies $D_{v_1v_3}\in\{{\rm S_P}\}_{\rm LA}$. Now we assume

\noindent{\it \textbf{Induction Hypothesis}}. For $l\leq K-1$ with $K>2$, the conclusion holds.

\noindent{\it Step 3.} Let $l=K$ and $|{\rm S_P}\mcap\mathpzc{C}|=q$. Since $q$ is odd, $q\geq 1$. We can assume that $C_{v_1v_2}\in {\rm S_P}$. For $\{v_2,v_3\}\in {\rm P}$, there are two possibilities: i) $B_{v_2v_3}\in{\rm S_P}$; ii) $C_{v_2v_3}\in{\rm S_P}$.
\begin{itemize}
  \item Proof under Case (i): if $B_{v_2v_3}\in{\rm S_P}$, then $C_{v_1v_3}\in\{{\rm S_P}\}_{\rm LA}$. Define ${\rm S_{P'}}:=\big({\rm S_P}\setminus \{C_{v_1v_2},B_{v_2v_3}\}\big)\mcup\{C_{v_1v_3}\}$ with ${\rm P'}=\big(\{v_1,v_3\},\dots,\{v_l,v_1\}\big)$. We must have $|{\rm S_{P'}}|=K-1$ and $|{\rm S_{P'}}\mcap\mathpzc{C}|=q$. By our induction hypothesis, the Lie algebra generated by ${\rm S_{P'}}$ contains at least one element in $\mathpzc{D}$. In addition, ${\rm S_{P'}}\subset\{{\rm S_P}\}_{\rm LA}$. The desired result follows straightforwardly.
  \item Proof of Case (ii): if $C_{v_2v_3}\in{\rm S_P}$, then $B_{v_1v_3}\in\{{\rm S_P}\}_{\rm LA}$ and $q\geq 3$.  Define ${\rm S_{P'}}:=\big({\rm S_P}\setminus \{C_{v_1v_2},C_{v_2v_3}\}\big)\mcup\{B_{v_1v_3}\}$ with ${\rm P'}=\big(\{v_1,v_3\},\dots,\{v_l,v_1\}\big)$. We must have $|{\rm S_{P'}}|=K-1$ and $|{\rm S_{P'}}\mcap\mathpzc{C}|=q-2$.
      Of course, $q-2$ is odd. The rest of the proof proceeds as in Case (i).
\end{itemize}
This proves the desired lemma.
\end{proof}
As in the proof of (ii), Lemma \ref{su.2} gives $\{{\rm S}\}_{\rm LA}=\su(n)$. We have thus proved the sufficiency.
\subsubsection*{D.2 Proof of Necessity  for Proposition \ref{su.1}}
Suppose the conditions in Proposition \ref{su.1} are not true. Then for the graph $\mathscr{G}_{\rm S}$, there are three possibilities: i) $\mathscr{G}_{\rm S}$ is not connected; ii) $\mathscr{G}_{\rm S}$ is a connected graph containing only Blue edges; iii) $\mathscr{G}_{\rm S}$ is a connected graph without self-loops, and there are no cycles with an odd number of Red edges. It is easy to check that when $\mathscr{G}_{\rm S}$ is not connected, the generated Lie algebra by ${\rm S}$ is not equal to $\su(n)$. If $\mathscr{G}_{\rm S}$ is a connected graph containing only Blue edges, then $\{{\rm S}\}_{\rm LA}=\so(n)$. Now let $\mathscr{G}_{\rm S}$ be a connected graph without self-loops, and there are no cycles with an odd number of Red edges. The proof is completed by showing that $\{{\rm S}\}_{\rm LA}\neq\su(n)$.

First, we establish the relationship between Lie bracket operations and transitive closure of the edge-colored multigraph $\mathscr{G}_{\rm S}$. Since ${\rm S}_3=\emptyset$, then ${\rm S}={\rm S_1}\mcup{\rm S_2}$. Let $\mathrm {S}^0\subseteq \mathrm {S}^1 \subseteq \mathrm {S}^2\cdots$ be an ascending chain of subsets of $\mathpzc{B}\mcup\mathpzc{C}$ such that $\mathrm {S}^0=\mathrm {S}, \mathrm {S}^1 = [\mathrm {S}^0,\mathrm {S}^0]\mcup \mathrm {S}^0, \dots, \mathrm {S}^{k+1} = [\mathrm {S}^k, \mathrm {S}^k]\mcup \mathrm {S}^k,\dots.$
where $[\mathrm {S}^k, \mathrm {S}^k] = \big\{[A,B]: A,B \in \mathrm {S}^k\big\}$. The following lemma holds.
\begin{lemma}
 $\mathcal{T}^{k}(\mathscr {G}_{\mathrm {S}})=\mathscr {G}_{\mathrm {S}^k}$ holds for all $k= 0, 1,\dots$.
\end{lemma}
\begin{proof}
It is evident that $\mathcal{T}^{0}(\mathscr {G}_{\mathrm {S}})=\mathscr{G}_{\mathrm {S}^0}$. Thus, in the rest of the proof, we focus on proving  $\mathcal{T}^{k}(\mathscr{G}_{\mathrm {S}})=\mathscr{G}_{\mathrm {S}^k}$ for $k\geq 1$. By the definition of $\mathcal{T}(\cdot)$
and ${\rm S}^k$, we conclude that $\mathcal{T}(\mathscr{G}_{\mathrm {S}^{k-1}})=\mathscr{G}_{\mathrm {S}^k}$ if and only if $\mathscr{G}_{\mathrm {S}^{k-1}}$ has no multiple edges. From now on, the notation $\mathscr{G}\rhd \mathcal {R}$ means that $\mathscr{G}$ does not have a cycle with an odd number of Red edges.

Recall that $\mathscr{G}_{\mathrm {S}}\rhd \mathcal {R}$. This implies that $\mathscr{G}_{\mathrm {S}}$ has no multiple edges, and thus $\mathcal{T}^{1}(\mathscr{G}_{\mathrm {S}})=\mathscr{G}_{\mathrm {S}^1}$. It can be concluded that $\mathscr{G}_{\mathrm {S}^1}$ has no multiple edges. If the assertion was false, then there must exist a cycle of length 3 containing an odd number of Red edges in the graph $\mathscr{G}_{\mathrm {S}}$. This leads to a contradiction. Therefore, $\mathcal{T}(\mathscr{G}_{\mathrm {S}^1})=\mathscr{G}_{\mathrm {S}^2}$, and, together with $\mathscr{G}_{\mathrm {S}^1}=\mathcal{T}^{1}(\mathscr{G}_{\mathrm {S}})$, we have $\mathcal{T}^{2}(\mathscr{G}_{\mathrm {S}})=\mathscr{G}_{\mathrm {S}^2}$.

Our task now is to prove that $\mathscr{G}_{\mathrm {S}^2}$ has no multiple edges. This result will be proved if we can show that $\mathscr{G}_{\mathrm {S}^1}\rhd \mathcal {R}$. To this end, consider $\mathscr{G}_{\mathrm {S}}$. There are two possibilities for $\mathscr{G}_{\mathrm {S}}$: i) Each cycle in $\mathscr{G}_{\mathrm {S}}$ contains an even number of Red edges (We adopt the convention that $0$ is even); ii) $\mathscr{G}_{\mathrm {S}}$ has no cycles.
\begin{itemize}
  \item Proof under Case (i): Let each cycle in $\mathscr{G}_{\mathrm {S}}$ contain an even number of Red edges. Note that $\mathscr{G}_{\mathrm {S}^1}=\mathcal{T}(\mathscr{G}_{\mathrm {S}})$. We only need to consider the newly generated cycles in $\mathscr{G}_{\mathrm {S}^1}$ compared to $\mathscr{G}_{\mathrm {S}}$. Let ${\rm P}=\big(\{v_1,v_2\},\{v_2,v_3\},\dots,\{v_l,v_1\}\big)$ be a cycle of $\mathscr{G}_{\mathrm {S}^1}$. Let us first assume that $\{v_1,v_2\}\in\mathscr{G}_{\mathrm {S}^1}, \{v_1,v_2\}\notin\mathscr{G}_{\mathrm {S}}$. For the colour of edge $\{v_1,v_2\}$, there are two possibilities (i) if  $\{v_1,v_2; {\rm Blue}\}\in\mathscr{G}_{\mathrm {S}^1}$, then there must exist a node $v_*$ such that $\{v_1,v_*; {\rm Blue}\}$ and $\{v_*,v_2; {\rm Blue}\}$ in $\mathscr{G}_{\mathrm {S}}$, or, $\{v_1,v_*; {\rm Red}\}$ and $\{v_*,v_2; {\rm Red}\}$ in $\mathscr{G}_{\mathrm {S}}$. Let ${\rm P'}=\big(\{v_1,v_*\},\{v_*,v_2\},\{v_2,v_3\},\dots,\{v_l,v_1\}\big)$. It is easily seen that the number of Red edges contained in ${\rm P'}$ and  ${\rm P}$ has the same parity; (ii) if $\{v_1,v_2; {\rm Red}\}\in\mathscr{G}_{\mathrm {S}^1}$, then there must exist a node $v_*$ such that $\{v_1,v_*; {\rm Red}\}$ and $\{v_*,v_2; {\rm Blue}\}$ in $\mathscr{G}_{\mathrm {S}}$. We see at once that ${\rm P'}$ and  ${\rm P}$ contain the same number of Red edges.

      Therefore, the parity of the number of Red edges contained in ${\rm P'}$ and ${\rm P}$ is always the same. It follows that the number of Red edges in ${\rm P}$ is even, if all edges in ${\rm P'}$ are in $\mathscr{G}_{\mathrm {S}}$. If not, we can use the above method to replace the edges of $\mathscr{G}_{\mathrm {S}^1}$ in the circle ${\rm P'}$ with the edges of $\mathscr{G}_{\mathrm {S}}$, until all the edges in the obtained cycle ${\rm P^*}$ are in $\mathscr{G}_{\mathrm {S}}$. Furthermore, the Red edges contained in ${\rm P}$ and ${\rm P^*}$ have the same parity. Then it can be concluded that the number of Red edges in ${\rm P}$ is even.

      On account of the arbitrariness of ${\rm P}$, we have $\mathscr{G}_{\mathrm {S}^1}\rhd \mathcal {R}$.

  \item Proof of Case (ii): if $\mathscr{G}_{\mathrm {S}}$ has no cycles, then the edge set of graph $\mathcal{T}(\mathscr{G}_{\mathrm {S}})$ is the union of triangles like $\big\{\{v_i,v_j;{\rm Blue}\},\{v_j,v_k;{\rm Blue}\},\{v_k,v_i;{\rm Blue}\}\big\}$ and $\big\{\{v_i,v_j;{\rm Red}\},\{v_j,v_k;{\rm Red}\},\{v_k,v_i;{\rm Blue}\}\big\}$. It can be easily verified that $\mathcal{T}(\mathscr{G}_{\mathrm {S}})\rhd \mathcal {R}$. This gives $\mathscr{G}_{\mathrm {S}^1}\rhd \mathcal {R}$.

\end{itemize}
We have thus proved $\mathcal{T}(\mathscr{G}_{\mathrm {S}})=\mathscr{G}_{\mathrm {S}^{1}}\rhd \mathcal {R}$. In fact, the proof above gives more, namely $\mathcal{T}^{k}(\mathscr{G}_{\mathrm {S}})\rhd \mathcal {R}$ holds for $k=0,1,2,...$. This certainly guarantees $\mathcal{T}^{k}(\mathscr{G}_{\mathrm {S}})$ has no multiple edges for all $k\geq1$, and, in consequence, the proof is completed by an easy induction.
\end{proof}

We are now in a position to show that $\{{\rm S}\}_{\rm LA}\neq\su(n)$. Note that $\mathscr{G}_{\mathrm {S}}$ has finitely many vertices and edges, the
ascending chain of graphs $\mathscr{G}_{\mathrm {S}}\subseteq\mathscr{G}_{\mathrm {S}^1}\subseteq\cdots$ stabilizes in finite steps. That is, there exists a nonnegative integer $z$ such that $\mathscr{G}_{\mathrm {S}^z}=\mathscr{G}_{\mathrm {S}^{z+1}}=\cdots$. In addition, $\mathscr{G}_{\mathrm {S}}$ is connected, and $\mathscr{G}_{\mathrm {S}^k}$ has no multiple edges for all $k\geq1$. It follows that $\mathscr{G}_{\mathrm {S}^z}$ is a complete graph with no multiple edges. From the definition of ${\rm S}^k$, we deduce that $\mathrm {S}^z$ spans $\{{\rm S}\}_{\rm LA}$.  Hence, $\{{\rm S}\}_{\rm LA}$ is a $n(n-1)/2$-dimensional subalgebra of $\su(n)$, and consequently  $\{{\rm S}\}_{\rm LA}\neq\su(n)$. We have now completed the proof of Proposition \ref{su.1}.

\subsection*{E. Proof of Theorem \ref{thm4}}

For any $A\in\su(n)$, we can represent it in the form of
$$
A=\sum_{k=1}^{l_1} a_{k} B_{{i}_k{j}_k}+\sum_{k=l_1+1}^{l_2} a_{k} C_{{i}_k{j}_k}+\sum_{k=l_2+1}^{l_3} a_{k} D_{{i}_k{j}_k}
$$
where $a_k\neq 0\in\mathbb{R}$, $i_k,j_k\in\mathrm{V}$, and $B_{{i}_k{j}_k}\in\mathpzc{B}, C_{{i}_k{j}_k}\in\mathpzc{C}, D_{{i}_k{j}_k}\in\mathpzc{D}$. Define the map $\psi$ that takes a matrix $A\in\su(n)$ to an edge-colored multigraph $\mathscr{G}_{A}:=(\mathrm{V},\mathscr{E}_{A})$ with $\mathscr{E}_{A}=\mathscr{E}_{\rm Blue}\mcup\mathscr{E}_{\rm Red}\mcup\mathscr{E}_{\rm Green}$, where $\mathscr{E}_{\rm Blue}=\big\{\{i_1,j_1;{\rm Blue}\},\dots,\{i_{l_1},j_{l_1};{\rm Blue}\}\big\}$, $\mathscr{E}_{\rm Red}=\big\{\{i_{l_1+1},j_{l_1+1};{\rm Red}\},\dots,\{i_{l_2},j_{l_2};{\rm Red}\}\big\}$, and $\mathscr{E}_{\rm Green}=\big\{\{i_{l_2+1},i_{l_2+1};{\rm Green}\},\dots,\{i_{l_3},i_{l_3};{\rm Green}\}\big\}$.

\begin{definition}\label{map.2}
Let $\mathrm{G}=(\mathrm{V},\mathrm{E})$ be an undirected graph. Given any node pair $(i,j)$ of $\mathrm{G}$, the graph
$\mathcal{H}_{ij}(\mathrm{G}) := (\mathrm{V},\mathrm{E}_{ij})$ is called the circumjacent closure at node pair $(i,j)$ of $\mathrm{G}$ with
$\mathrm{E}_{ij}=\mathrm{E}_{ij}^{1}\mcup\mathrm{E}_{ij}^{2}$, where
\begin{align*}
\mathrm{E}_{ij}^{1}&= \big\{\{i,k\}\,:\,\exists k \ {\rm s.t. }\ \{j,k\}\in \mathrm{E}\big\}, \quad\\
&~~~~ \mathrm{E}_{ij}^{2}= \big\{\{j,k\}\,:\,\exists k \ {\rm s.t. }\ \{i,k\}\in \mathrm{E}\big\}.
\end{align*}
\end{definition}

\begin{lemma}\label{lem4}
 Suppose the graph $\mathrm{G}=(\mathrm{X\mcup Y},\mathrm{E})$ is a bi-graph with $|\mathrm{X}|\geq3,|\mathrm{Y}|\geq3$ and $|\mathrm{E}|\geq1$. Then there exists  a finite sequence of node pairs $(i_1,j_1),\dots,(i_z,j_z)$ for some integer $z\geq 1$ such that

$(i)$ Either $i_{s},j_{s}\in \mathrm{X}$ or $i_{s},j_{s}\in \mathrm{Y}$ for $s=1,2,\dots,z$;

$(ii)$ $\mathcal{H}_{i_{z}j_{z}}\big(\cdots\mathcal{H}_{i_{2}j_{2}}\big(\mathcal{H}_{i_{1}j_{1}}(\mathrm{G})\big)\big):=\big(\mathrm{X\mcup Y},\mathrm{E}_{i_{z}j_{z}}\big)$ is also a bi-graph with $|\mathrm{E}_{i_{z}j_{z}}|=1$.
\end{lemma}
The proof of this lemma appears in \cite{arxiv}, and therefore we will omit it. In general, we have the following result, and it can be proved in much the same way as Lemma \ref{lem4}.
\begin{lemma}\label{lem5}
Let $\mathrm{G}:=(\mathrm{V},\mathrm{E})$ be an undirected graph with $|\mathrm{E}|\geq1$. Then there exists  a finite sequence of node pairs $(i_1,j_1),\dots,(i_z,j_z)$ for some integer $z\geq 1$, such that the cardinality of the edge set of graph $\mathcal{H}_{i_{z}j_{z}}\big(\cdots\mathcal{H}_{i_{2}j_{2}}\big(\mathcal{H}_{i_{1}j_{1}}(\mathrm{G})\big)\big)$ is equal to 1.
\end{lemma}

If $\mathscr{G}$ is an edge-colored multigraph without self-loops and multiple edges, then $\mathscr{G}$ can be regarded as an undirected graph without considering the color of the edges, and the conclusions in Lemma \ref{lem4} and \ref{lem5} are still valid for $\mathscr{G}$. The proof of Theorem \ref{thm4} is based on the following lemma.

\begin{lemma}\label{lem6}
Consider a subset $\mathrm {S}\subseteq\mathpzc{B}\mcup\mathpzc{C}\mcup\mathpzc{D}$ with the associated multigraph $\mathscr{G}_{\mathrm {S}}=(\mathrm{V},\mathscr{E}_{\mathrm {S}})$ and a matrix $A\in\su(n)$ with the associated multigraph $\mathscr{G}_A=(\mathrm{V},\mathscr{E}_A)$. Let $\mathscr{G}_{\mathrm {S}}$ be connected. Then the Lie algebra generated by $\{A\}\mcup\mathrm {S}$ is equal to $\su(n)$ if and only if $\mathscr{G}_A\mcup\mathscr{G}_{\rm S}$  has a self-loop or a cycle with an odd number of Red edges.
\end{lemma}

\begin{proof}
(Sufficiency) Since $\mathscr{G}_{\mathrm {S}}$ is connected, Proposition \ref{su.1} shows that the Lie algebra generated by ${\rm S}$ is equal to $\su(n)$ if $\mathscr{G}_{\rm S}$ has a self-loop or a cycle with an odd number of Red edges. Now, assume that $\mathscr{G}_{\rm S}$ has neither a self-loop nor a cycle containing an odd number of Red edges. Analysis similar to that in the proof of necessity  for Proposition \ref{su.1} shows that there exists a nonnegative integer $z$ such that $\mathscr{G}_{\mathrm {S}^z}=\mathscr{G}_{\mathrm {S}^{z+1}}=\cdots$. In addition, $\mathscr{G}_{\mathrm {S}^z}=({\rm V}, \mathscr {E}_{\mathrm {S}^z})$ is a complete graph with no self-loops or multiple edges. It follows that for any $1\leqslant i< j\leqslant n$, we have $F_{ij}\in\mathrm {S}^z$, either $F=B$ or $F=C$. Below we respectively prove that $\big\{\{A\}\mcup {\rm S}\big\}_{\rm LA}=\su(n)$ under the condition that $\mathscr{G}_A\mcup\mathscr{G}_{\rm S}$ has a self-loop or a cycle with an odd number of Red edges. The proof falls naturally into two parts.

Let us first show that if $\mathscr{G}_A\mcup\mathscr{G}_{\rm S}$ has self-loops, i.e., $\mathscr{G}_{A}$ has self-loops, then  $\big\{\{A\}\mcup {\rm S}\big\}_{\rm LA}=\su(n)$. Define
$$\widetilde{A}:=A-\sum_{\{{i}_k,{j}_k;{\rm Blue}\}\in \mathscr{E}_{\mathrm {S}^z}}a_{k}B_{{i}_k{j}_k}-\sum_{\{{i}_k,{j}_k;{\rm Red}\}\in \mathscr{E}_{\mathrm {S}^z}}a_{k}C_{{i}_k{j}_k}-\sum_{\{{i}_k,{i}_k;{\rm Green}\},\{{j}_k,{j}_k;{\rm Green}\} \in\mathscr{E}_{\mathrm {S}^z}}a_{k}D_{{i}_k{j}_k}.$$
It is immediate that $\widetilde{A}\neq0$ and $\mathscr{G}_{\widetilde{A}}=({\rm V},\mathscr{E}_{\widetilde{A}})$ has self-loops. Moreover, we have	$\big\{\{A\}\mcup{\rm S}\big\}_{\rm LA}=\big\{\{\widetilde{A}\}\mcup{\rm S}^z\big\}_{\rm LA}$. Without loss of generality we can assume $\widetilde{A}=A_{*}+aD_{12}$, $a\neq 0\in\mathbb{R}$. Consider the graph $\mathscr{G}_{\widetilde{A}}=({\rm V},\mathscr {E}_{\widetilde{A}})$. There are two possibilities (i) $\{1,2\}\notin\mathscr {E}_{\widetilde{A}}$; (ii) $\{1,2\}\in\mathscr {E}_{\widetilde{A}}$.
\begin{itemize}
	\item Proof under Case (i): if $\{1,2\}\notin\mathscr {E}_{\widetilde{A}}$, then
	$$\psi\big([\widetilde{A},F_{12}]\big):=\mathscr{G}=({\rm V},\mathscr {E})$$
	has neither self-loops nor multiple edges, and $|\mathscr {E}|\geq 1$. With the notation $\widehat{A}=[\widetilde{A},F_{12}]$, we have $\mathscr{G}=\psi(\widehat{A})$. Applying Lemma \ref{lem5}, there exists  a finite sequence of node pairs $(i_1,j_1),\dots,(i_w,j_w)$ for some integer $w\geq 1$, such that the cardinality of the edge set of graph $\mathcal{H}_{i_{w}j_{w}}\big(\cdots\mathcal{H}_{i_{2}j_{2}}\big(\mathcal{H}_{i_{1}j_{1}}(\mathscr{G})\big)\big)$ is equal to one. Let $\mathcal{H}_{i_{w}j_{w}}\big(\cdots\mathcal{H}_{i_{2}j_{2}}\big(\mathcal{H}_{i_{1}j_{1}}(\mathscr{G})\big)\big)=\big({\rm V},\big\{\{i^*,j^*\}\big\}\big)$. We conclude that
	\begin{align*}
		\psi\big([F_{i_{w}j_{w}},\dots,[F_{i_{2}j_{2}},[\widehat{A},F_{i_{1}j_{1}}]]]\big)
		=\big({\rm V},\big\{\{i^*,j^*\}\big\}\big).
	\end{align*}
	Hence 
	\begin{equation*}
		[F_{i_{w}j_{w}},\dots,[F_{i_{2}j_{2}},[\widehat{A},F_{i_{1}j_{1}}]]]=a^{*}F_{i^{*}j^{*}},
	\end{equation*}
	where $a^{*}$ is the coefficient generated during the operation of the Lie brackets. Since $F_{i_sj_s}\in{\rm S}^z$ for $s=1,\dots,w$, we have $F_{i^{*}j^{*}}\in\big\{\{\widetilde{A}\}\mcup{\rm S}^z\big\}_{\rm LA}$. In addition, with the definition of $\widetilde{A}$, we have $F_{i^{*}j^{*}}\notin{\rm S}^z$. Therefore, $\mathscr{G}_{\{F_{i^{*}j^{*}}\}\mcup{\rm S}^z}$ has  multiple edges. Proposition \ref{su.1} now yields $\big\{\{F_{i^{*}j^{*}}\}\mcup{\rm S}^z\big\}_{\rm LA}=\su(n)$. This implies that $\big\{\{\widetilde{A}\}\mcup{\rm S}^z\big\}_{\rm LA}=\su(n)$, and thus $\big\{\{A\}\mcup{\rm S}\big\}_{\rm LA}=\su(n)$.
	\item Proof of Case (ii): in  $\mathscr{G}_{\widetilde{A}}$, the nodes adjacent to node $1$ except node $2$ are denoted as $i_1,\dots,i_m$, and, the nodes adjacent to node $2$ except node $1$ are denoted as $j_1,\dots,j_n$, where $m,n\geq0$.
	If $m=0$, let $i\in{\rm V}$ and $i\neq1\neq2$. Then it is evident that graph $\psi\big([\widetilde{A},F_{1i}]\big)$ has neither self-loops nor multiple edges, and the cardinality of its edge set is greater than or equal to two. This is due to the fact that $\{1,i\}\notin\mathscr {E}_{\widetilde{A}}$. If $m\geq 1$, then the edge set of $\psi\big([\widetilde{A},F_{12}]\big)$ contains edge $\{2,i_1\}$. Let $i_*\in{\rm V}$ and $i_*\neq i_1\neq1\neq2$. We see at once that edge $\{i_1,i_*\}$ is not in the edge set of $\psi\big([\widetilde{A},F_{12}]\big)$. Therefore, $\psi\big([[\widetilde{A},F_{12}],F_{i_1i_*}]\big)$ has neither self-loops nor multiple edges, and the cardinality of its edge set is greater than or equal to one.
	
	In summarizing, there must exist a matrix $\widehat{A}\in\big\{\{\widetilde{A}\}\mcup{\rm S}^z\big\}_{\rm LA}$ such that $\psi(\widehat{A})$ has neither self-loops nor multiple edges, and the cardinality of its edge set is greater than or equal to one. Analysis similar to that in the proof of Case (i) shows that $\big\{\{A\}\mcup{\rm S}\big\}_{\rm LA}=\su(n)$.
\end{itemize}

The proof above gives more, namely $\big\{\{\widetilde{A}\}\mcup{\rm S}^z\big\}_{\rm LA}=\su(n)$ if $\widetilde{A}\neq0$.

We next prove that if $\mathscr{G}_A\mcup\mathscr{G}_{\rm S}$ has a cycle with an odd number of Red edges, then $\big\{\{A\}\mcup {\rm S}\big\}_{\rm LA}=\su(n)$.  We continue to use the definition of $\widetilde{A}$ in the above. Now $\widetilde{A}\neq 0$, which is due to the fact that $\mathscr{G}_{\mathrm {S}^z}$ dose not have a cycle with an odd number of Red edges, while $\mathscr{G}_A\mcup\mathscr{G}_{\rm S}$ has. It follows easily that $\big\{\{A\}\mcup{\rm S}\big\}_{\rm LA}=\big\{\{\widetilde{A}\}\mcup{\rm S}^z\big\}_{\rm LA}=\su(n)$.

(Necessity) Assume that $\mathscr{G}_A\mcup\mathscr{G}_{\rm S}$ has neither self-loops nor a cycle with an odd number of Red edges. Proposition \ref{su.1} shows that the Lie algebra generated by ${\rm S}$ is not equal to $\su(n)$. In addition, $\mathscr{G}_{\mathrm {S}^z}$ is complete, and it has neither a self-loop nor a cycle with an odd number of Red edges. We must have $\mathscr{E}_A\subseteq \mathscr{E}_{\mathrm {S}^z}$, because $\mathscr{G}_{A}\mcup\mathscr{G}_{{\rm S}^z}$ has neither a self-loop nor a cycle with an odd number of Red edges. This implies that $\big\{\{A\}\mcup{\rm S}\big\}_{\rm LA}=\{{\rm S}\}_{\rm LA}\neq\su(n)$. Therefore, if the Lie algebra generated by $\{A\}\mcup\mathrm {S}$ is equal to $\su(n)$, then $\mathscr{G}_A\mcup\mathscr{G}_{\rm S}$  must have a self-loop or a cycle with an odd number of Red edges. We have now completed the proof of Lemma \ref{lem6}.
\end{proof}

We are now in a position to present the proof of Theorem \ref{thm4}. Let $\mathscr{G}_{\rm contr}^{{}^\beta}$ be connected. Since $\mathscr{G}_{\rm drift}^{{}^\alpha} \mcup \mathscr{G}_{\rm contr}^{{}^\beta}$ has a self-loop or a cycle with an odd number of Red edges, there must exist a matrix $A\in\Sigma_{\rm r}(\su(n)_{\alpha})$ such that $\mathscr{G}_A\mcup \mathscr{G}_{\rm contr}^{{}^\beta}$ has a self-loop or a cycle with an odd number of Red edges. Using Lemma \ref{lem6}, we see that $\big\{\{A\}\mcup\so(n)_\beta\big\}_{\rm LA}=\su(n)$. From this we conclude that the system \eqref{bilinear} is structurally controllable on the Lie group ${\rm SU}(n)$ with respect to the pair of zero patterns $\Sigma_{\rm zero}:=\big(\Sigma_{\rm r}(\su(n)_{\alpha}),\Sigma_{\rm f}(\su(n)_{\beta})\big)$.

\subsection*{F. Proof of Theorem \ref{thm5}}
The statement in Theorem \ref{thm5} will be proved once we prove the lemma below.
\begin{lemma}\label{su.3}
Consider a subset $\mathrm {S}\subseteq\mathpzc{B}\mcup\mathpzc{C}\mcup\mathpzc{D}$ with the associated multigraph $\mathscr{G}_{\mathrm {S}}=(\mathrm{V},\mathscr{E}_{\mathrm {S}})$ and a matrix $A\in\su(n)$ with the associated multigraph $\mathscr{G}_A=(\mathrm{V},\mathscr{E}_A)$. The Lie algebra generated by $\{A\}\mcup\mathrm {S}$ is equal to $\su(n)$ if the following conditions hold:

$(i)$ Each connected component of $\mathscr{G}_{\mathrm {S}}$ contains at least three nodes;

$(ii)$ $\mathscr{G}_{A}$ has no multiple edges and the union graph $\mathscr{G}_A \mcup\mathscr{G}_{\mathrm{S}}$ is connected;

$(iii)$ $\mathscr{G}_A \mcup\mathscr{G}_{\mathrm{S}}$  has a self-loop or a cycle with an odd number of Red edges.
\end{lemma}

\begin{proof}
If $\mathscr{G}_{\rm S}$ is connected, then by Lemma \ref{lem6}, the Lie algebra generated by $\{A\}\mcup\mathrm {S}$ is equal to $\su(n)$. Now assume that $\mathscr{G}_{\rm S}$ has $m$ connected components with $ m\geq2$. Let $\mathscr{G}_{i}=(\mathrm{V}_{i},\mathscr{E}_{i})$ denote the $i$-th connected component of  $\mathscr{G}_{\rm S}$ for $i=1,\dots,m$. We continue to use the definition of ${\rm S}^k$ in the proof of Proposition \ref{su.1}. Write $\overline{{\rm S}}=\bigcup_{k=1}^\infty {\rm S}^k$. There must exist a nonnegative integer $z$ such that ${\rm S}^z ={\rm S}^{z+1}=\dots$, which then implies $\overline{{\rm S}}={\rm S}^z$. By the definition of $\overline{{\rm S}}$, we see that ${\rm S}^z$ spans $\{{\rm S}\}_{\rm LA}$. Since $\mathscr{G}_{\rm S}$ has $m$ connected components, $\mathscr{G}_{{\rm S}^z}$ has $m$ connected components too. In addition, for any $u<v\in\mathrm{V}_{i}$, $B_{uv}\in{\rm S}^z$ or $C_{uv}\in{\rm S}^z$.

We continue to use the definition of $\widetilde{A}$ in the proof of Lemma \ref{lem6}. As the union graph $\mathscr{G}_{A} \mcup \mathscr{G}_{{\rm S}^z}$ is connected while $\mathscr{G}_{{\rm S}^z}$ is not, we always have $\mathscr{E}_{\widetilde{A}}\neq\emptyset$, i.e., $\widetilde{A}\neq 0$.
This allows us to further conclude that
\begin{equation}\label{eq7}
\big\{\{A\}\mcup{\rm S}\big\}_{\rm LA}=\big\{\{\widetilde{A}\}\mcup{\rm S}^z\big\}_{\rm LA}.
\end{equation}
The remainder of the proof is divided into two steps.

\noindent{\it Step 1.} We first prove the statement when $\mathscr{G}_{\rm S}$ contains only two connected components, i.e., $m=2$.
Because $\mathscr{G}_{A} \mcup \mathscr{G}_{\rm S}$ is connected, $\mathscr{G}_{\widetilde{A}} \mcup \mathscr{G}_{{\rm S}^z}$ is also connected by the definition of $\mathscr{G}_{\widetilde{A}}$ and $\mathscr{G}_{{\rm S}^z}$. In addition, there must exist an edge $\{u,v\}\in\mathscr{G}_{\widetilde{A}}$ such that $u\in{\rm V}_1,v\in{\rm V}_2$. Let $v_{11}$ be the node in $\mathrm{V}_{1}$ with $\deg(v_{11})=k>0$, and $v_{i_{1}j_{1}},\dots,v_{i_{k}j_{k}}$ be its neighbors. For convenience, we let $v_{i_{1}j_{1}},\dots,v_{i_{r}j_{r}}\in\mathrm{V}_{2}$, $1\leq r\leq k$. Consider the node $v_{12}\in \mathrm{V}_{1}$. We must have $F_{v_{11}v_{12}}\in {\rm S}^z$, where $F=B$ or $F=C$.

Let $v_{i_{*}j_{*}}\neq v_{i_{1}j_{1}}\in \mathrm{V}_{2}$. There holds $F_{v_{i_{1}j_{1}}v_{i_{*}j_{*}}}\in {\rm S}^z$. Lemma \ref{lem3} now shows that
$$\psi\big([[\widetilde{A},F_{v_{11}v_{12}}],F_{v_{i_{1}j_{1}}v_{i_{*}j_{*}}}]\big):=\mathscr{G}=\big({\rm V}_1\mcup{\rm V}_2, \mathscr {E}\big)$$
is a bi-graph, with $|\mathscr {E}|\geq1$. Write $[[\widetilde{A},F_{v_{11}v_{12}}],F_{v_{i_{1}j_{1}}v_{i_{*}j_{*}}}]=\widehat{A}$. It is immediate that $\mathscr{G}=\psi(\widehat{A})$.
From Lemma \ref{lem4}, there is an integer $z\geq 1$ and a finite sequence of node pairs either $v_{s},u_{s}\in \mathrm{V}_{1}$ or $v_{s},u_{s}\in \mathrm{V}_{2}$ for $s=1,2,\dots,z$,
such that $\mathcal{H}_{v_{z}u_{z}}\big(\cdots(\mathcal{H}_{v_{1}u_{1}}(\mathscr{G}))\big)=\big(\mathrm{V}_{1}\mcup\mathrm{V}_{2},\mathscr{E}_{v_{z}u_{z}}\big)$ is also a bi-graph with $|\mathscr{E}_{v_{z}u_{z}}|=1$. Let $\mathscr{E}_{v_{z}u_{z}}=\big\{\{v^{*},u^{*}\}\big\}$, where $v^{*}\in\mathrm{V}_{1},u^{*}\in\mathrm{V}_{2}$.
 Note that $\{v_{s},u_{s}\}\in\mathscr{E}_{{\rm S}^z}$, we have $F_{v_{s}u_{s}}\in {\rm S}^z$ for $ s=1,\dots,z$.

Based on Definition \ref{map.2} and Lemma \ref{lem3}, it can be concluded
 \begin{align*}
\psi\big([F_{v_{z}u_{z}},\dots,[F_{v_{2}u_{2}},[\widehat{A},F_{v_{1}u_{1}}]]]\big)
 =\mathcal{H}_{v_{z}u_{z}}\big(\cdots(\mathcal{H}_{v_{1}u_{1}}(\mathscr{G}))\big)
 \end{align*}
without considering the color of the edges.
This yields
 $$[F_{v_{z}u_{z}},\dots,[F_{v_{2}u_{2}},[\widehat{A},F_{v_{1}u_{1}}]]]=a^{*}F_{v^{*}u^{*}},$$
where $a^{*}$ is the coefficient generated during the operation of the Lie brackets. Therefore, $F_{v^{*}u^{*}}\in\big\{\{\widetilde{A}\}\mcup{\rm S}^z\big\}_{\rm LA}$, and, together with the connectivity of $\mathscr{G}_{\{F_{v^{*}u^{*}}\}\mcup{\rm S}^z}$, Lemma \ref{lem6} now leads to
$$\big\{\{\widetilde{A}\}\mcup{\rm S}^z\big\}_{\rm LA}=\big\{\{\widetilde{A}\}\mcup\{F_{v^{*}u^{*}}\}\mcup{\rm S}^z\big\}_{\rm LA}=\su(n).$$
By \eqref{eq7} it is obvious that $\big\{\{A\}\mcup{\rm S}\big\}_{\rm LA}=\su(n).$

\noindent{\it Step 2.} In this step, we proceed to establish the result for the general case by induction on the number of connected components of $\mathscr{G}_{\rm S}$.

\noindent{\it \textbf{Induction Hypothesis}}. If graph $\mathscr{G}_{\rm S}$ contains $m\geq 2$ connected components, then $\big\{\{A\}\mcup{\rm S}\big\}_{\rm LA}=\su(n)$. We will prove it for $m+1$.

To do this, consider $\mathscr{G}_{\widetilde{A}}$. Let $v_{11}$ be the node in $\mathrm{V}_{1}$ with $\deg(v_{11})=k>0$, and $v_{i_{1}j_{1}},\dots,v_{i_{k}j_{k}}$ denote the nodes that are adjacent to $v_{11}$. There is no loss of generality in assuming $v_{i_{1}j_{1}},\dots,v_{i_{r}j_{r}}\in \mathrm{V}_{2}$, with $1\leq r\leq k$. Analysis similar to that in the step 1 shows that
$$\psi\big([[\widetilde{A},F_{v_{11}v_{12}}],F_{v_{i_{1}j_{1}}v_{i_{*}j_{*}}}]\big)=\psi(\widehat{A})=\mathscr{G}_a\mcup\mathscr{G}_b,$$
where $\mathscr{G}_a=({\rm V}_1\mcup{\rm V}_2,\mathscr {E}_a)$ is a bi-graph with $|\mathscr {E}_a|\geq1$, $\mathscr{G}_b=(\mcup_{i=3}^{m+1}{\rm V}_i,\mathscr {E}_b)$ is a empty graph. Again we apply Lemma \ref{lem4} to $\mathscr{G}_a$ and obtain $F_{v^{*}u^{*}}\in\big\{\{\widetilde{A}\}\mcup{\rm S}^z\big\}_{\rm LA}$ with $v^{*}\in\mathrm{V}_{1},u^{*}\in\mathrm{V}_{2}$. It follows immediately that $F_{v^{*}u^{*}}\in\big\{\{A\}\mcup{\rm S}\big\}_{\rm LA}$. Now $\mathscr{G}_{\{F_{v^{*}u^{*}}\}\mcup{\rm S}}$ has $m$ connected components, and by our induction hypothesis, $\big\{\{A\}\mcup\{F_{v^{*}u^{*}}\}\mcup{\rm S}\big\}_{\rm LA}=\su(n)$. Thus we arrive at the conclusion that $\big\{\{A\}\mcup{\rm S}\big\}_{\rm LA}=\su(n)$, and the proof is complete.
\end{proof}

\end{document}